\documentclass[onefignum,onetabnum]{siamart171218}

\usepackage{lipsum}
\usepackage{amsfonts}
\usepackage{graphicx}
\usepackage{epstopdf}
\usepackage{algorithmic}
\usepackage{amsopn}

\usepackage[utf8x]{inputenc}
\usepackage[T1]{fontenc}
\usepackage{amsmath,amstext,amssymb,amscd,amsfonts}
\usepackage[english]{babel}
\usepackage{enumitem}
\usepackage{url}
\usepackage{mathtools}
\usepackage{esint}
\mathtoolsset{centercolon}
\usepackage{dsfont}
\usepackage{nicefrac}
\usepackage{suffix}
\usepackage{color}
\usepackage{scrlayer-scrpage}

\newcommand\norm[1]{\left\lVert#1\right\rVert}
\WithSuffix\newcommand\norm*[1]{\lVert#1\rVert}
\newcommand\set[1]{\left\{#1\right\}}
\WithSuffix\newcommand\set*[1]{\{#1\}}
\newcommand\bra[1]{\left({#1}\right)}
\WithSuffix\newcommand\bra*[1]{({#1})}
\newcommand\sca[1]{\left\langle{#1}\right\rangle}
\WithSuffix\newcommand\sca*[1]{\langle{#1}\rangle}
\newcommand\abs[1]{\left\lvert#1\right\rvert}
\WithSuffix\newcommand\abs*[1]{\lvert#1\rvert}

\newcommand{\dx}[1]{\;\de\! #1}

\newcommand{\vphi}{\varphi}

\newcommand{\ls}{\lesssim}

\newcommand{\N}{\mathbb{N}}
\newcommand{\Z}{\mathbb{Z}}
\newcommand{\R}{\mathbb{R}}
\newcommand{\one}{\mathds{1}}

\newcommand{\dH}{\dot{H}^1}
\newcommand{\dHinv}{\dot{H}^{-1}}
\newcommand{\dHs}{\dot{H}_{\sigma}^1}
\newcommand{\dHsinv}{\dot{H}_{\sigma}^{-1}}
\newcommand{\ignore}[1]{}

\DeclareMathOperator{\Id}{Id}

\DeclareMathOperator{\Div}{div}
\DeclareMathOperator{\dist}{dist}

\DeclareMathOperator{\skw}{skew}
\DeclareMathOperator{\eff}{eff}

\DeclareMathOperator{\de}{d}

\newtheorem{thm}{Theorem}[section]

\newtheorem{rem}[thm]{Remark}
\newtheorem{cor}[thm]{Corollary}

\numberwithin{equation}{section}

\newsiamremark{remark}{Remark}
\newsiamremark{hypothesis}{Hypothesis}
\crefname{hypothesis}{Hypothesis}{Hypotheses}
\newsiamthm{claim}{Claim}

\headers{A local version of Einstein's formula}{B. Niethammer and R. Schubert}

\title{A local version of Einstein's formula for the effective viscosity of suspensions}

\author{Barbara Niethammer\thanks{Institute for applied mathematics, University of Bonn, Endenicher Allee 60, 53115 Bonn, Germany, 
  (\email{niethammer@iam.uni-bonn.de}).}
\and Richard Schubert\thanks{Lehrstuhl 1 f\"ur Mathematik, RWTH Aachen, Pontdriesch 14-16, 52056 Aachen, Germany, 
  (\email{schubert@math1.rwth-aachen.de}).}}

\ifpdf
\hypersetup{
  pdftitle={A local version of Einstein's formula for the effective viscosity of suspensions},
  pdfauthor={B. Niethammer, R. Schubert}
}
\fi

\begin{document}

\maketitle

{\centering
\today \\
}

\begin{abstract}
  We prove a local variant of Einstein's formula for the effective viscosity of dilute suspensions, that is $\mu^\prime=\mu \bra{1+\frac 5 2\phi+o(\phi)}$, where $\phi$ is the volume fraction of the suspended particles. Up to now rigorous justifications have only been obtained for 
  dissipation functionals of the flow field. We prove that the formula holds on the level of the Stokes equation (with variable viscosity). 
  We consider a regime where the number $N$ of particles suspended in the fluid goes to infinity while their size $R$ and the volume fraction $\phi=NR^3$ approach zero.
  We  establish $L^\infty$ and $L^p$ estimates for the difference of the microscopic solution to the solution of the homogenized equation.
  Here we assume that the particles  are contained in a bounded region and are well separated in the sense that the minimal distance is comparable to the average one. 
  The main tools for the proof are a dipole approximation of the flow field of the suspension together with the so-called method of reflections and a coarse graining 
  of the volume density.
\end{abstract}



\section{Introduction}

In his annus mirabilis, 1905, Einstein published five seminal works contributing to different areas of physics. One of these works was his dissertation
"Eine neue Bestimmung der Molek\"uldimensionen"  \cite{Ein06} where he derives a formula for the effective viscosity of a dilute suspension of spheres 
\begin{equation}\label{eq:Einstein}
 \mu^\prime=\mu \Big(1+\frac 5 2\phi+o(\phi)\Big), 
\end{equation}
where $\phi$ is the (small) volume fraction of the spheres. He relates it to the formula for the mass diffusivity in order to obtain a formula for the size of the particles in the suspension. Applying this to a solution of sugar in water he is able to estimate the molecular dimensions of sugar, since both viscosity and diffusivity can be measured experimentally.

Today the formula and its validity is still of interest because dilute suspensions appear in a wide scope of applications in physics, chemistry and engineering where the effective properties of such mixtures play a role. 

The purpose of this article is to rigorously prove Einstein's formula on a local level. So instead of considering global energies of the flow field of 
the suspension we prove that the flow obeys an effective equation incorporating the (possibly non-constant) effective viscosity at {every} point in space. 
This is an improvement with regard to prior results because
it makes the time evolution of the problem accessible and it provides a proof of Einstein's formula in a non-periodic setting.
 
\subsection{Setting of the problem}\label{subsec:heuristics}

Consider a collection of rigid spherical particles $B_i^N\coloneqq B_{R^N}(X_i^N),i=1,..,N$ where $X_i^N\in \R^3$ are the particles' centres and $R^N>0$ is
the radius of all particles. Let 
\begin{equation*}
d^N_{ij}\coloneqq \abs{X^N_i-X^N_j} \;\;\text{ for }i\neq j, \quad \quad d^N\coloneqq \min_{1\le i,j\le N}d^N_{ij}>2R^N.
\end{equation*}
 This implies that the particles do not intersect nor touch each other. The domain of the suspending material is given by 
\begin{equation*}
\Omega^N=\R^3\setminus \bigcup_{i=1}^N \overline{B^N_i}.
\end{equation*} 
For an easier reading we will mostly drop the superscript $N$  in the further discussion.

We assume that  $\Omega$ is occupied by a Stokes-fluid with viscosity  $\mu=1$, that the particles are inertialess and that the fluid-solid interaction
is given by a no-slip boundary condition. This entails the following problem for the fluid velocity $u:\R^3\to \R^3$:
\begin{align}
 \Delta u+\nabla p&=f \quad \text{in }\Omega,\label{eq:s1}\\
 \Div u&=0 \quad \text{in }\Omega,\label{eq:s2}\\
 \int_{\partial B_i} \sigma n\dx{S}&=0 \quad \text{, }i=1,\dots,N,\label{eq:s3}\\
 \int_{\partial B_i} (x-X_i)\wedge(\sigma n)\dx{S}&=0 \quad \text{, }i=1,\dots,N,\label{eq:s4}\\
 u(x)&=V_i+\omega_i\wedge (x-X_i ) \quad \text{on } \overline{B_i} \;\;\text{, } i=1,\dots,N,\label{eq:s5}\\
 u(x)&\to 0 \quad \text{as }\abs{x}\to \infty,\label{eq:s6}
\end{align}
where 
\begin{align*}
 \sigma=-p\Id+2 eu, \quad eu=\frac 1 2 (\nabla u+\nabla u^T),
\end{align*}
 and the $V_i,\omega_i \in \R^3$ are a priori unknown and must be determined as part of the solution. 
 Here $p$ is the pressure which is a Lagrange multiplier associated to the divergence condition and $f$ is a force density acting on the fluid. In problem \eqref{eq:s1}-\eqref{eq:s6} one can replace $f$ by $f^N=\one_\Omega f$ where $\one_\Omega$ is the characteristic function of $\Omega$ since the equation holds only in $\Omega$. 


\subsubsection{Heuristics}
We start by a heuristic derivation of the effective flow field. To that aim we consider the solution without particles: 
\begin{align}
 -\Delta v+\nabla p&=f^N \quad \text{ in }\R^3,\label{eq:v1}\\
 \Div v&=0 \quad \text{in }\R^3,\label{eq:v2}\\
 v(x)&\to 0 \quad \text{as }\abs{x}\to \infty.\label{eq:v3}
\end{align}
Since we are in a low volume fraction regime it is reasonable (and will be proven in Theorem \ref{thm:u_v1_diff}) to assume that $u$ is already close, in terms of the volume fraction $\phi$, to $v$. In fact it is possible to prove $\norm{u-v}\ls \phi$ for a suitable norm. Then, in order to get a higher-order approximation of $u$, the main point is to satisfy the condition $u=V_i+\omega_i\wedge (x-X_i)$ on the particles. On the ball $B_i$ the function $v$, up to first order, has the form 
\begin{equation*}
v(x)=v(X_i)+\nabla v(X_i)(x-X_i)+o(R).
\end{equation*}
The linear part consists of a skew-symmetric part that induces rotations and that we want to keep, while we need to correct for the symmetric part $\epsilon_i=ev(X_i)$. 
In order to get closer to a rigid body motion, we subtract from $v$ the (dipole-)function $d_i$ that only incorporates the symmetric gradient and is defined by
\begin{align}\label{eq:di}
 d_i(x)=\begin{cases}
         \epsilon_i(x-X_i) \quad &\mbox{, for } x\in \overline{B_i},\\
         \frac 5 2 R^3\frac{(x-X_i)\bra{(x-X_i)\cdot\epsilon_i(x-X_i)}}{\abs{x-X_i}^5}\\
         +R^5\bra{\frac{\epsilon_i(x-X_i)}{\abs{x-X_i}^5}-\frac 5 2  \frac{(x-X_i)\bra{(x-X_i)\cdot\epsilon_i(x-X_i)}}{\abs{x-X_i}^7}} &\mbox{, otherwise.}
        \end{cases}
\end{align}
Then $v-d_i=v(X_i)+\omega_i \wedge (x-X_i)+o(R)$ in $B_i$, where $\omega_i$ is determined by the skew-symmetric part of the gradient. Now we want the higher-order approximation $\tilde{u}$ to be close to a rigid body motion on all the particles and set 
\begin{equation}\label{eq:tildeu}
\tilde{u}=v-\sum_{i=1}^N d_i.
\end{equation}
Of course for $i\neq j$ the dipole $d_i$ will not vanish on $B_j$ but since the decay of $d_i$ is quadratic we may hope that under some conditions on the particle distribution this effect is comparable to the one coming from higher order terms in the Taylor expansion of $v$ in $B_i$. A related approach would be to choose, instead of $\epsilon_i$, the coefficients of the dipoles in order to optimize with respect to the distance of $\tilde{u}$ to $u_N$, see for example \cite{MM10} for the scalar case, \cite{MMN16} for the vectorial case in the framework of elasticity and the review article \cite{MM18} as well as references therein. We refrain from taking this approach here because it makes the second (homogenization) step harder. Since the $d_i$ solve the homogeneous Stokes equation outside $\overline{B_i}$ the equation $-\Delta \tilde{u}+\nabla p=f^N$ is valid in $\Omega$. Note that
$d_i$ consists of two parts, one of which decays much more rapidly than the other. Hence we take into account only the first part for the following heuristics. 
Now let $\phi=RN^3$ (this is a slight abuse of notation since $\phi$ denoted the physical volume fraction before). We assume that the rescaled volume density $\rho^N=\frac 1 \phi \frac{4\pi}{3} R^3 \sum_{i=1}^N\delta_{X_i}$, where $\delta_{X_i}$ is the dirac measure supported at $X_i$, converges in some sense to $\rho$  as $N\to\infty$ so that $\phi \rho$ is the virtual limit volume density. We can write
\begin{align*}
 \tilde{u}(x)&=v(x)-\sum_{i=1}^N d_i(x)\\
 &\approx v(x)-\int_{\R^3}\frac{3}{4\pi} \phi \rho^N(y) \frac 5 2 \bra{\frac{(x-y)\bra{(x-y)\cdot ev(y)(x-y)}}{\abs{x-y}^5}}\dx{y}.
\end{align*}
Now we introduce the fundamental solution to the Stokes equation 
\begin{equation*}
\Phi_{ij}(x)=\frac{1}{8\pi}\bra{\frac{\delta_{ij}}{\abs{x}}+\frac{x_ix_j}{\abs{x}^3}}. 
\end{equation*}
We will see later (part 3 of the proof of Lemma \ref{lemma:tu_hu_diffS}) that the following identity holds for symmetric and traceless matrices $\epsilon$, where here, and in the following we use the Einstein convention to always sum over doubly appearing indices:
\begin{align*}
 \epsilon_{ki} \partial_k \Phi_{ij}(x)=-\frac{3}{8\pi}\frac{x_jx_k \epsilon_{ki}x_i}{\abs{x}^5}.
\end{align*}
Using this we arrive at the following approximation:
\begin{align*} 
 \tilde{u}_j(x)&\approx v_j(x)+\int_{\R^3}5 \phi \rho^N(y)  ev(y)_{ki}\bra{\partial_k \Phi_{ij}}(x-y)\dx{y}\\
 &\approx v_j(x)+\int_{\R^3} 5 \phi \rho(y) ev(y)_{ki}\bra{\partial_k \Phi_{ij}}(x-y)\dx{y}\\
 &= v_j(x)+\int_{\R^3} \Phi_{ij}(x-y)\Div_y\bra{5 \phi \rho ev(y)}_i\dx{y}.
\end{align*}
Notice that $\rho^N\rightharpoonup \rho$ implies that $\frac 1 \phi \one_\Omega \rightharpoonup \rho$. This yields $f^N\approx (1-\phi \rho)f$ for large $N$. Using this and taking $-\Delta$ on both sides we arrive at
\begin{align*}
 -\Delta \tilde{u} +\nabla p&=(1-\phi \rho)f+\Div\bra{5 \phi \rho ev}.
\end{align*}
We expect that $\tilde{u}$ is a better approximation of $u$ than $v$, in particular, since $\norm{u-v}\sim \phi$, we have  $\|\tilde{u}-v\| \sim \phi$. Thus, making an error of order $\phi^2$, we can replace $v$
by $\tilde u$ in the divergence term to obtain the following equation
\begin{align*}
 -\Div\bra{\nabla \tilde{u}+5\phi\rho e \tilde{u}}+\nabla \tilde{p}=(1-\phi \rho)f,\quad \Div \tilde{u}=0.
\end{align*}
We can use the fact that $\Div \tilde{u}=0$ (and hence $\Div \nabla\tilde{u}^T=0$) to write this equation as
\begin{align}\label{eq:final}
 -\Div\bra{\bra{2+5\phi\rho} e \tilde{u}}+\nabla \tilde{p}=(1-\phi \rho)f,\quad \Div \tilde{u}=0.
\end{align}
This has the form
\begin{equation*}
-\Div \sigma=(1-\phi \rho)f, \quad\text{ where }\quad \sigma=2\Big(1+\frac 5 2 \phi\rho\Big)e\tilde{u}-p\Id.
\end{equation*}
Comparing to the form of the stress tensor for a homogeneous fluid (and keeping in mind that we rescaled the viscosity) this suggests that the effective viscosity of a suspension for small volume fractions of the immersed particles in a material of viscosity $\mu$ is given by 
\begin{equation*}
\mu_{\eff}=\mu \Big(1+\frac 5 2\phi\rho\Big)
\end{equation*}
to first order in $\phi$. Note that, since $\rho$ is typically non-constant, the effective viscosity is a function of the space variable. 

In regions where the density $\rho$ is constant, the divergence acting on the part of the transposed gradient vanishes because $\Div \tilde{u}=0$. In these regions we recover Einstein's formula even for the classical form of the Stokes equation:
\begin{align*}
 -\Big(1+\frac 5 2 \phi \rho\Big)\Delta \Tilde{u}+\nabla \tilde{p}&=(1-\phi \rho)f,\\
 \Div \tilde{u}&=0.
\end{align*} 
The main result of this paper shows, informally stated, that $\frac{u-\tilde u}{\phi}$ goes to zero as $\phi \to 0$ in suitable norms. This result is part of the second author's PhD thesis \cite{Sch19}. Some standard results that are not mentioned in this article and extended proofs can be found in \cite{Sch19}.

The same computation as above can be done for the analogous but simpler electrostatical problem of perfectly conducting spheres suspended in a dielectric medium. The corresponding result is that the effective permittivity is given by $\eta_{\eff}=\eta (1+3\phi\rho)$. We also refer to \cite{Sch19} for details. 
 
\section{The main result}

In this section we state our assumptions on the particle configuration and give a precise formulation of our main result.

\subsection{Assumptions}
We set $\phi=NR^3$ and assume that the following requirements are met by the sequence of particle configurations:
\begin{align}
&\mbox{There exists $L>0$ such that } \abs{X_i}+R<L \mbox{ for all } i=1,\dots,N. \label{assump1}\\
&\mbox{There exists  } C>0 \mbox{ such that  } N^{-\frac 1 3}\le Cd. \label{assump2}\\
&\mbox{We assume that  } \phi \log N\to 0 \mbox{ as } N\to \infty. \label{assump3}  
\end{align}
Note that \eqref{assump2} and \eqref{assump3} imply that the particles are well separated in the sense that $d \ge 4R$ for large $N$. We need assumption \eqref{assump3} which is a bit stronger than the minimal assumption $\phi \to 0$ as $N\to \infty$ to account for the logarithmic divergence of certain sums and integrals at the origin.

We will make the following assumptions on $f$:
\begin{align}
 f\in L^{\frac 6 5}(\R^3) \cap L^{\infty}(\R^3)\label{assump4}\\
 f \in  C^{0,\alpha}(\R^3) \mbox{ for some } \alpha>0. \label{assump5}
 \end{align}

In order to state the result that compares the microscopic solution of problem \eqref{eq:s1}-\eqref{eq:s6} to the solutions
of certain homogenized problems, it is necessary to define a limit volume density. 
It will prove useful to introduce a coarse grained density as in \cite{NV06}.

\begin{definition}
 Let $s^N>0$ be a sequence such that $s^N\log N\to 0$ as $N\to \infty$.
  Let $\R^3$ be decomposed into half-open 
 disjoint cubes $A_j$ of side length $s^N$ where $j\in \Z^3$. 
We define the rescaled averaged particle volume density $\rho^N$ by
\begin{align}\label{eq:rhoN}
 \rho^N(x)=\frac{4\pi}{3}\frac{1}{N(s^N)^3}n(A_j) \qquad \mbox{ for } x \in A_j,
\end{align}
where $n(A_j)$ is the number of particle centers $X_i$ in $A_j$.
\end{definition}

Note that $\phi\rho^N$ is the local volume density of the particles in each cube. By assumption \eqref{assump2} this vanishes in the limit $N\to \infty$ which forces us to rescale by the volume fraction, in order to obtain a quantity proportional to the number density, that does not necessarily converge to zero. Since all particles are contained in a big ball (assumption \eqref{assump1}), $\rho^N$ will, for large $N$ be compactly supported in $B_{L+1}(0)$. By assumption \eqref{assump2} we have that  $\rho^N$ is uniformly bounded in $L^\infty$. For the following we will assume without loss of generality that 
\begin{equation}\label{assump6}
 \rho^N \rightharpoonup \rho \quad \mbox{ weakly* in } L^{\infty}(\R^3) \mbox{ and in all } L^p(\R^3), p\geq 1\,.
\end{equation}
In addition we need the regularity assumption  
\begin{equation}\label{assump7}
\rho\in W^{1,\infty}(\R^3)
\end{equation}
which is needed, because in equation \eqref{eq:final} the derivative coming from the divergence might fall onto the (non-constant) density. This produces a lower order term that can only be controlled with assumption \eqref{assump7}.

As the dipoles used for the approximation of the microscopic problem are singular, we are forced to define a domain that 
cuts out a boundary layer around the particles. Let $\delta^N>0$ such that $\frac 1 {\bra{\delta^N}^2N}\to 0$ and $\frac{\delta}{d}\to 0$ as $N\to \infty$. 
In particular $N^{-\frac 1 2}\le C \delta^N\le C^\prime N^{-\frac 1 3}$. 
Then we define $r=\max(2R,\delta^N)$ and introduce
\begin{equation}\label{eq:omegadelta}
\Omega_\delta^N=\R^3\setminus \cup_{i=1}^N B_r(X_i).
\end{equation}  

\subsection{Statement of main result}\label{Ss.statement}

For the following let $\dH$ be the closure of functions in $C^\infty_c(\R^3,\R^3)$ with respect to the $L^2$ norm of the gradient and let $\dHinv$ be its dual.
In order to incorporate the incompressibility condition we define
$ \dHs=\{w\in \dH: \Div w=0\}$ 
and denote its  dual by $\dHsinv$.

Our main result is the following local version of Einstein's formula.
\begin{thm}\label{thm:mainS} 
 The solution $\bar{u}\in \dHs$ to the equation 
 \begin{align}
  -\Div\bra{\bra{2+5 \phi \rho}e\bar{u}}+\nabla p&=(1-\phi \rho)f,\label{eq:homu}\\
  \Div \bar{u}&=0,
 \end{align}
 is close to $u$, the solution of (\ref{eq:s1})-(\ref{eq:s6}), in the following sense:
 \begin{align*}
  \frac{1}{\phi}\norm{u-\bar{u}}_{L^\infty(\Omega_\delta^N)}\to 0, N\to \infty.
 \end{align*}
  Furthermore, if $U\subset \R^3$ is of finite measure and $p\in [1,\frac 3 2]$, then
  \begin{align*}
  \frac 1 \phi \norm{u-\bar{u}}_{L^p(U)}\to 0, N\to \infty.
 \end{align*}
\end{thm}

\begin{rem}\label{rem:asymp}
 Note that
\begin{equation*}
\abs{\R^3\setminus \Omega^N_\delta}=\frac{4\pi}{3}Nr^3\le C \max(2R,\delta)^3\le CNd^3 \max\Big(\frac{R}{d},\frac{\delta}{d}\Big)\to 0, N\to \infty.
\end{equation*} 
Therefore, even in $L^\infty$, the solution of the homogenized equation is close to $u$ on scale $\phi$ in a volume that is asymptotically the whole $\R^3$.

Since Einstein's formula is only asymptotic for $\phi\to 0$ and the effective viscosity at $\phi=0$ is $\mu$ there is no hope to obtain convergence to the solution of a single non-trivial limit problem. Therefore the comparison to a family of limiting problems is the right approach. Note that already the difference of $u$ to the homogeneous solution with visocsity $\mu$ is of order $\phi$. So only after dividing the difference by $\phi$ the result is meaningful.
\end{rem}

\subsection{Prior and related results}

While the number of rigorous mathematical papers regarding Einstein's formula is limited, there is a multitude of physics literature. 
As a first generalization ellipsoidal particles have been considered by Jeffery \cite{Jef22} and later by 
Hinch and Leal \cite{LH71, HL72}, while drops of  another fluid (with finite viscosity) suspended in a surrounding fluid are considered for the first time
in \cite{Tay32}.

In \cite{KRM67} the authors establish several extremum principles for the Stokes flow including fairly general boundary conditions and rigid particles. They use those principles to obtain, among other results, bounds and an asymptotic formula for the effective viscosity in the low concentration regime and for high concentrations when the particles are situated on a lattice. In the same year in \cite{FA67} another result is given for high concentrations. Numerical research can be found in \cite{NK84}, in which arbitrary concentrations are considered and also asymptotic formulas for high concentrations are obtained. \cite{BBP05} considers the case of highly concentrated suspensions and uses a so-called network approximation. 

In \cite{KRM67} the effective viscosity is obtained by comparing the dissipation rate of the suspension and of a homogenous fluid with different viscosity subject to pure strain boundary conditions. The pure strain boundary condition is imposed for a domain that becomes infinite in the limit in order to circumvent boundary effects. This disadvantage is overcome only in 2012 by Haines and Mazzucato \cite{HM12} when they rigorously prove, simultaneously bounding the power of the next order term, that
\begin{equation*}
 \abs{\mu_{\eff}-\mu\bra{1+\frac 5 2 \phi}}\le C\mu \phi^{\frac 3 2}.
\end{equation*}
They consider a fixed domain with pure strain boundary conditions with particle positions fixed to a lattice and also compare dissipation rates. 
 
 In \cite{LS85} the periodic homogenization of the Navier-Stokes equation is discussed. For the first time, the effective viscosity is not determined by an asymptotic or a dissipation functional method, but as a prefactor of the strain in the homogenized equation. In their paper the authors derive a homogenized Navier-Stokes equation up to terms of order $\phi$ that includes the term (\cite[p. 13]{LS85})
\begin{equation*}
 \Div\bra{2\bra{1+\frac 5 2 \phi \rho}eu}.
\end{equation*}
Almog and Brenner \cite{AB98} consider non-constant volume fraction and ensemble averages and obtain an effective viscosity field $\mu(x)$ which confirms 
Einstein's formula. Also here the effective viscosity appears inside the Stokes equation. They also recover the results up to $\phi^2$ with a second factor $6.95$. 
Although \cite{LS85} and \cite{AB98} derive equations with an effective viscosity and are in that respect similar to our approach, both results are not completely rigorous.

A second order correction to the viscosity is first considered by Batchelor and Green. In \cite{BG72} they calculate the second order correction to the viscosity for a random distribution of spheres to be $7.6\phi^2$ with an estimated error of the numerical factor of $10\%$ which comes from numerical and asymptotic evaluation of an, in principle, known function. \cite{AGKL12} obtain the term $\frac 5 2 \phi^2$ (the numerical value is computed wrongly in the paper despite a correct formula) and the recent \cite{GH19} confirm this result (see also below).

Shortly after the first submission of our article, several parallel contributions appeared, that together yield a substantially clearer picture of the matter than the one that existed during the writing of this article. In \cite{HW19} Einstein's formula is justified in a quantitative way and for general shapes with controlled diameter. There one can also find a hint on how to generalize our method to general shapes. The idea is to compare the dissipation energy of the shape to the one of an enclosing sphere and derive estimates from there. \cite{GH19} contains a conditional result on the second order correction and in particular they recover the term $\frac 5 2 \phi^2$ obtained in \cite{AGKL12} (apparently the paper contains a miscalculation and $\frac 5 2$ is the corrected result) in the case of periodically arranged spheres. \cite{DG19} considers the more abstract problem of a general effective viscosity tensor for finite volume fraction, i.e. they consider the limit $N\to \infty$ without simultaneously letting $\phi\to 0$. They derive, in the stationary and ergodic random setting a formula similar to the one obtained in \cite{MK06} for the corresponding electrostatic problem, at the same time providing corrector results. Finally, \cite{Ger19} gives a short and clean proof for a result that justifies the analogue of Einstein's formula in the electrostatic setting (and seems to work similarly for the Stokes problem). Let us stress here that although quantifying convergences and considering more general cases in comparison to this article, all results in \cite{GH19,HW19,Ger19,DG19} provide closeness results in low $L^p$ spaces $p\le 2$. However, in order to apply the results to the dynamic problem and to consider the evolution of the particle density it is crucial to have estimates in $L^\infty$. We provide those in the whole space for the relevant approximation (Theorem \ref{thm:u_v1_diff}) and for a domain that converges to $\R^3$ in case of the homogenized velocity field (Remark \ref{rem:asymp}).


\section{Strategy of the proof and preliminaries}

\subsection{Strategy of proof }

The proof is divided into two parts. In the first part we justify the dipole approximation while in the second part we make the heuristic computations of Subsection \ref{subsec:heuristics} rigorous. To that end we use successive approximations 
$u\to v_1\to \tilde{u}\to \hat{u}\to \bar{u}$.

In Section \ref{sec:dipole} we prove that the explicit dipole approximation $\tilde{u}$ from (\ref{eq:tildeu}) is actually close to the microscopic solution $u$.
For this we first define a related but abstract dipole approximation $v_1$ defined via projections to subspaces of $\dHs$ incorporating the 
rigid body boundary conditions on the particles (Subsection \ref{subsec:abstract_dipoles}). The so-called method of reflections then gives closeness of $v_1$ to $u$ 
(see Theorem \ref{thm:u_v1_diff}). The method of reflections for several particles was first introduced by Smoluchowski in \cite{Smo11} and is used extensively in the physics literature.
Rigorous proofs for the convergence of the method suited for the treatment of sedimenting particles are given in \cite{Luk89,JO04,HV18,Hof18}. See also \cite{MNP00,MMP02} where a reflection method is used for a single particle in a domain with boundary. We adapt here the method developed in \cite{Hof18} for dipole approximations.
Since we take into account rotations for the particles in the Stokes case, some adjustments have to be made and it is necessary to establish a Korn and a Korn-Poincar\'e inequality for balls with integrated boundary conditions. 
Using carefully obtained characterizations of the projections (Subsection \ref{subsec:characterization}) we then show that inside the particles $v_1$ and $\tilde{u}$ are already close and that, again, using the decay of the dipoles, this can be extended to the domain outside the particles (see Subsection \ref{subsec:convergence}).

Closeness of $v_1$ to $\tilde{u}$ (see Lemma \ref{lemma:v1_u_diff} and Lemma \ref{lemma:Lpdiff}) is a consequence of the fact that both the abstract and the explicit dipoles and hence also their difference are dipoles and thus have good decay properties. This is shown in Subsection \ref{subsec:explicit}. 

The second part of the proof consists in proving the closeness of $\tilde{u}$ to the solution $\bar{u}$ of the Stokes equation with Einstein viscosity (\ref{eq:homu}). This is done in Section \ref{sec:homogenization}. We use an intermediate approximation $\hat{u}$ which is the solution to the equation
\begin{equation}\label{eq:homv}
 -\Div\bra{\nabla \hat{u}+5\phi \rho ev}+\nabla p=(1-\phi \rho)f. 
\end{equation}
To prove that $\tilde{u}$ is close to $\hat{u}$ (Lemma \ref{lemma:tu_hu_diffS}) we use the fact that for every point in space the contributions of the particles in a moderately large region around this point are negligible. But further away the number density $\rho^N$ looks approximately like $\rho$ which allows passage from sum to integral. The proof relies heavily on the representation of solutions as convolutions with the fundamental solutions and involves various estimates regarding these convolution integrals. 

In order to replace $v$ by $\bar{u}$ in equation \eqref{eq:homv} we first prove that $v$ is already close to $\bar{u}$ namely that
$ \norm{v-\bar{u}}\le C \phi$.
This is achieved by standard regularity arguments and estimates of the solutions of the homogenized equation in terms of the right hand side. With the same methods it is then possible to prove that the solution to \eqref{eq:homv}, $\hat{u}$, is close to the solution of the final equation \eqref{eq:homu}, $\bar{u}$, since their difference satisfies an equation with a right hand side that is already small (Lemma \ref{lemma:hu_bu_diff}).

The strategy of first providing a suitable approximation to the microscopic solution from which the homogenized problem can be derived was already executed in several works. See for example \cite{Hof18} in the case of sedimentation or \cite{AZ17} in the case of acoustic waves.

\subsection{Notation}\label{subsec:notation}

In addition to $\dH$ and $\dHs$ introduced above in Section \ref{Ss.statement} we introduce some further notation.
Here for all  spaces we will drop the target space $\R^3$ in the notation. 
 We denote the $\dH$ pairing with $\sca{\cdot,\cdot}$
while we write $\bra{\cdot,\cdot}$ for the $L^p-L^q$ pairing where $\frac 1 p+\frac 1 q=1$. For two functions $u,v\in \dH$ this means
$ \sca{u,v}=\bra{\nabla u,\nabla v}$.
For $f\in \dHinv$ we write write $(f,\vphi)$ for $f[\vphi]$ which coincides with the classical notation if $f\in L^{\frac 6 5}(\R^3)$. 

 We denote the symmetric gradient of $w$ by $ew=\frac 1 2 \bra{\nabla w+\nabla w^T}$. Note that for functions in $\dHs$ the $L^2$ pairing of the gradients is the same up to a factor of $2$ as the $L^2$ pairing of the symmetric gradients.

Also note that there is a unique vector $\omega u(x) \in \R^3$ such that $\nabla u(x) \,y=eu(x)\,y+\omega u(x)\wedge y$ for all $y\in \R^3$. This is just a consequence of the fact that for every skewsymmetric matrix $S\in \R^{3\times 3}, S=-S^T$ there exists $\omega\in \R^3$ such that $Sy=\omega\wedge y$ for all $y\in\R^3$.

By the Riesz theorem, for any $f \in \dHsinv$ there is a $w\in \dHs$ such that 
\begin{equation}\label{eq:scaS}
\bra{f,\vphi}=\sca{w,\vphi} \mbox{ for all }\vphi\in \dHs.
\end{equation} 
By \cite[Lemma V.1.1]{Gal94} we have that if equation \eqref{eq:scaS} holds for all $\vphi\in \dHs$, then, there exists $p\in L^2(\R^3)$ such that 
\begin{equation}
  \bra{f,\vphi}=\sca{w,\vphi}+\bra{\Div \vphi,p}\mbox{ for all }\vphi\in \dH.
\end{equation}    
We then say that 
\begin{equation}
  -\Delta w+\nabla p=f,
\end{equation}
in the weak sense. The solution operator $S^{-1}:\dHsinv\to\dHs$ that maps $f$ to $w$ is an isometric isomorphism and its inverse $S$ is the so-called Stokes operator. 

The solution operator $S^{-1}$ is given by $S^{-1}f=\Phi\ast f$ where $\Phi$ is the fundamental solution of the Stokes equation, the so-called Oseen tensor
\begin{align*}
 \Phi(x)=\frac{1}{8\pi}\bra{\frac{\Id}{\abs{x}}+\frac{x\otimes x}{\abs{x}^3}}.
\end{align*}

The corresponding pressure such that $-\Delta S^{-1}f+\nabla p=f$ is given by
\[ p=\Pi\ast f \qquad \mbox{ with } \qquad  
 \Pi(x)=\frac{1}{4\pi}\frac{x}{\abs{x}^3}.
\]
Since the pressure $p$ is merely a Lagrange multiplier ensuring that the velocity field is solenoidal we will write $p$ for every appearing pressure, so that it may change between different equations but also from line to line in one computation. 

In the following universal constants $C>0$ will often appear in statements. They never depend on $N,R,d$ and $X_1,\dots,X_N$ and other $N$-dependent quantities but possibly on $f$ unless otherwise stated. When constants appear they might change their value from line to line without indication.

For all spaces we will use a $0$ as subscript to indicate that the boundary values of that function  vanish. 
Also, for any classical Sobolev spaces, the subscript $\sigma$ indicates that the weak divergence vanishes.

\subsection{Weak formulation of the problem}\label{subsec:formulation}

Let $f\in L^{\frac 6 5}(\R^3)\cap L^2(\R^3)$. A function $u$ is a weak solution of problem \eqref{eq:s1}-\eqref{eq:s6} if $u\in \dHs$ (which implies \eqref{eq:s2},\eqref{eq:s6}), if $u$ is a rigid body motion on all $\overline{B_i}$ for $i=1,..,N$ (this is \eqref{eq:s5}), if for all $\vphi \in \dH_{\sigma,0}(\Omega)$
\begin{equation}\label{eq:weakS}
 \int_{\Omega} \nabla u \cdot\nabla \vphi\dx{x}=\int_{\Omega} f\cdot \vphi \dx{x},
\end{equation}
 and if \eqref{eq:s3}, \eqref{eq:s4} are satisfied. Here $\nabla u$ and $p$ are a priori only in $L^2(\Omega)$ and \eqref{eq:s3}, \eqref{eq:s4} may seem a bit ambiguous at first glance since $\sigma$ is only in $L^2(\Omega)$ a priori. However, we have that $\Div \sigma=f^N\in L^2$ and hence $\sigma n\in H^{-\frac 1 2}(\partial B_i)$. That $\Div \sigma\in L^2(\Omega)$ is a simple consequence of \eqref{eq:weakS} and the so-called reciprocal principle (see, e.g. \cite{HB65}): For any $p\in L^2(\R^3)$ and $w\in \dHs$ we write $\sigma=2 ew-p\Id$. Then for $v,w\in \dHs$ we have 
\begin{align*}
  \int_{\R^3}\nabla w\cdot\nabla v\dx{x}=2 \int_{\R^3}ew\cdot ev\dx{x}= 2\int_{\R^3} ew\cdot \nabla v\dx{x}= \int_{\R^3} \sigma\cdot\nabla v\dx{x}.
\end{align*}
In the second step we used that the scalar product of a symmetric and a skew-symmetric matrix is zero ($\nabla v=ev+\bra{\nabla v}^{\skw}$), while in the second, we used, that $v$ is divergence-free whence $\Id\cdot \nabla v=\Div v=0$. The name reciprocal principle comes from the fact that the same equality holds for interchanged $w,v$. Note that, if $w$ satisfies \eqref{eq:s5}, then, because $ew=0$ in $B_i$ for all $i=1,\dots,N$ we can write
\begin{equation}\label{eq:recip}
  \int_{\R^3}\nabla w\cdot\nabla v\dx{x}=\int_{\Omega} \sigma\cdot\nabla v\dx{x}.
\end{equation} 
%

Existence of a weak solution to \eqref{eq:s1}-\eqref{eq:s6} follows by minimizing the energy
\begin{align}\label{eq:energy}
 E(w)=\int_{\R^3} \bra{\abs{ew}^2-f^N\cdot w} \dx{x},
\end{align}
in the space of functions that are rigid body motions inside the particles:
\begin{align}\label{eq:W}
 W\coloneqq \set{w\in \dHs: \exists\; V,\omega \in \bra{\R^3}^N\forall\; i: w(x)=V_i+\omega_i\wedge (x-X_i) \text{ on }\overline{B_i}}. 
\end{align}

\section{The dipole approximation}\label{sec:dipole}

It is useful not to start with the explicit dipole approximation from \eqref{eq:tildeu}, but to consider 
dipoles which can be characterized using a variational formulation which makes the comparison to the microscopic solution simpler. 
In order to do so we adapt the theory developed in \cite{Hof18}. With the exception of the Korn and Korn-Poincar\'e inequality in Lemma \ref{lemma:korn} and Corollary \ref{cor:korn_poincare}, which are not needed in \cite{Hof18}, all statements in Subsections \ref{subsec:abstract_dipoles}, \ref{subsec:characterization} and \ref{subsec:convergence} concerning the abstract dipoles are statements/ideas from \cite{Hof18} adapted to the situation of rigid body motions. Since the proofs are very similar we just state deviations and refer the reader to \cite{Hof18} and \cite{Sch19} for the detailed proofs.

\subsection{Approximation by abstract dipoles}\label{subsec:abstract_dipoles}
The solution $v$ of the particle-free problem \eqref{eq:v1}-\eqref{eq:v3} is given by $v=\Phi\ast f^N$ and is the minimizer of the energy $E$ from \eqref{eq:energy} in the space $\dHs$.

On the other hand, the solution $u$ to problem \eqref{eq:s1}-\eqref{eq:s6} is the minimizer of $E$ in the space $W$ (see equation \eqref{eq:W}). The fact that $u$ minimizes $E$ in $W$ means that $u$ is the orthogonal projection of $v$ from $\dHs$ to the subspace $W$. We call $P:\dHs\to W$ the orthogonal projection so that $u=Pv$. This implies that 
\begin{align*}
 \norm{v-u}_{\dH}\le \norm{v-w}_{\dH} \mbox{ for all }w\in W.
\end{align*}
By choosing a suitable function $w$ one can thus get an estimate for $\norm{u-v}_{\dH}$. This is part of the proof of Theorem \ref{thm:u_v1_diff}.

In order to control the velocity field locally we need to consider the $L^\infty$ norm, though. To get $L^\infty$ estimates it is useful to work with the method of reflections. This is due to the fact that the projection onto $W$ is not so easy to characterize. The method of reflections works with solutions of single particle problems. The single particle spaces involved, are much easier to characterize than $W$. For this we first define the particle wise version of $W$:
\begin{align*}
 W_i&=\set{w\in \dHs: w=V+\omega\wedge (x-X_i) \mbox{ on } \overline{B_i},\; V,\omega\in \R^3}.
\end{align*}
Since $W_i$ is a closed subspace of $\dHs$ there is an orthogonal projection $P_i:\dHs\to W_i$. Notice that $W=\cap_{i=1}^N W_i$. The orthogonal complement of $W_i$ has a useful characterization:

\begin{lemma}\label{lemma:perpspace}
\begin{align*}
 \bra{W_i}^\perp=&\left\{s\in \dHs: -\Delta s+\nabla p=0 \text{ in } \R^3\setminus \overline{B_i}, \; \int_{\partial B_i} \sigma[s] n\dx{S}=0,\right. \\ &\left. \int_{\partial B_i} (x-X_i)\wedge \bra{\sigma[s] n}\dx{S}=0\right\}.
\end{align*}
\end{lemma}

\begin{proof}
Take any element $s$ of the right hand side. Then for any $w\in W_i$, using \eqref{eq:recip}, we obtain (let $w=V+\omega\wedge (x-X_i)$ on $\overline{B_i}$):
  \begin{align*}
  \sca{w,s}&=\int_{\R^3\setminus \overline{B_i}} \nabla w\cdot \sigma[s] \dx{x}=-\int_{\partial B_i} w\bra{\sigma[s]n}\dx{S}-\int_{\R^3\setminus \overline{B_i}} w \Div\sigma[s] \dx{x}\\
  &=-\int_{\partial B_i} \bra{V+\omega\wedge (x-X_i)}\bra{\sigma[s]n}\dx{S}\\
  &=-V\int_{\partial B_i} \bra{\sigma[s]n}\dx{S}-\omega \int_{\partial B_i} (x-X_i)\wedge \bra{\sigma[s]n}\dx{S}=0.
 \end{align*}
On the other hand, since the first line is zero for any $s\in \bra{W_i}^\perp$ we have $0=-\Div \sigma[s]=-\Delta s+\nabla p$ in $\R^3\setminus \overline{B_i}$ by considering $w\in \dH_{\sigma,0}(\R^3\setminus \overline{B_i})$. Then we can conclude by using $V,\omega=e_1,e_2,e_3$ in the last line.
\end{proof}

A function with the property that 
\begin{equation*} 
 0=\int_{\partial B_i} \sigma n\dx{S}=\int_{\overline{B_i}} \Div \sigma \dx{x},
\end{equation*}
is usually called a dipole since the first moment of the force distribution vanishes inside the ball.
 
We come back to our goal to approximate $u$ by $v$. We already know that, following the idea from Subsection \ref{subsec:heuristics}, it makes sense to subtract from $v$ at every ball the dipole preventing $v$ from being a rigid body motion . Let $Q_i=\Id-P_i$ be the orthogonal projection onto $W_i^\perp$. We know that $v-Q_iv=P_i v\in W_i$ is a rigid body motion on the ball $B_i$, hence $Q_iv$ is the dipole we are looking for. As explained in Subsection \ref{subsec:heuristics}, subtracting $Q_iv$ only 
helps with the boundary condition on $B_i$, so we have to subtract the dipole for all balls which gives rise to the first approximation 
$ v_1=\big(\Id -\sum_{i=1}^N Q_i\big)v$
Of course this approximation will not be constant on all balls, since the additional $Q_jv$ for $j\neq i$ will have non-vanishing contributions on $B_i$. But since the dipoles solve the homogeneous Stokes equation outside $B_i$, the function $v_1$ still satisfies the same equation as $v$ and $u$ in $\Omega$. In order to get even closer to $u$ let us repeat the process of subtracting dipoles in order to make the functions closer to rigid body motions on the balls. This leads to approximations $v_k$ given by
\begin{equation}\label{eq:vk}
v_k=\Big(\Id -\sum_{i=1}^N Q_i\Big)^kv.
 \end{equation}
The idea is that taking $k\to \infty$ one should have 
 $P=\lim_{k \to \infty}\Big(\Id -\sum_{i=1}^NQ_i\Big)^k$.
 This would imply that $v_k\to Pv=u$ as $k\to \infty$. 

 Before attempting to prove that the $v_k$ converge to $u$ we need a better understanding of the projections $P_i$ and $Q_i$ respectively.

\subsection{Characterization of $W_i^\perp$ and $P_i$}\label{subsec:characterization}

\begin{lemma}\label{lemma:projectionsS}
 For $w\in W_i^\perp$ we have 
 \begin{equation}\label{eq:meanS}
 \fint_{\partial B_i} w\dx{S}=0 \quad \text{ and }\quad  \fint_{\partial B_i} (x-X_i)\wedge w\dx{S}=0.
 \end{equation}
  Hence, for $w\in \dHs$ the projection to $W_i$ satisfies $P_iw (x)=V+\omega\wedge (x-X_i)$ for all $x\in \overline{B_i}$ with 
  \begin{equation}\label{eq:vomegaS}
 V=\fint_{\partial B_i} w\dx{S} \quad \text{ and }\quad\omega=\frac{3}{2R^2} \fint_{\partial B_i} (x-X_i)\wedge w\dx{S}.
 \end{equation}
\end{lemma}

\begin{rem}
 Here and in the following we set
 \begin{equation*}
  \fint_{\partial B_i} w\dx{S}=\frac{1}{4\pi R^2}\int_{\partial B_i} w\dx{S}.
 \end{equation*}
\end{rem}

\begin{proof}
The idea is to use explicit test functions, for spatial motion and rotation separately, for which we know the drag on the sphere explicitly. See \cite{Hof18}, Lemma 3.10, for the first parts of \eqref{eq:meanS} and \eqref{eq:vomegaS}.

For the second part of \eqref{eq:meanS} take $\vphi$ such that $-\Delta \vphi+\nabla p=0$, $\Div \vphi=0$ in $\R^3\setminus \overline{B_i}$ and $\vphi=\omega\wedge (x-X_i)$, $\omega \in \R^3$ on $\overline{B_i}$. Then $\vphi \in W_i$ and 
\begin{equation*}
  \vphi(x)=R^3\frac{\omega\wedge (x-X_i)}{\abs{x-X_i}^3},
 \end{equation*} 
for $\abs{x-X_i}>R$. The drag on the sphere is then given by $(\sigma[\vphi] n)=-\frac{3}{R}\omega\wedge (x-X_i)$. Then 
 \begin{align*}
  0&=\sca{w,\vphi}=-\int_{\partial B_i} w\cdot \bra{\sigma[\vphi]n}\dx{S}-\int_{\R^3\setminus \overline{B_i}} w \cdot \Div\sigma[\vphi] \dx{x}\\
  &=\frac{3}{R}\int_{\partial B_i} w\cdot \bra{\omega\wedge (x-X_i)}\dx{S}=\frac{3}{R}\omega\cdot\int_{\partial B_i} (x-X_i)\wedge w\dx{S}.
 \end{align*}
 Setting $\omega=e_1,e_2,e_3$ we arrive at the second part of \eqref{eq:meanS}.
 
 Now for $w\in \dHs$ we know that $w-P_iw=Q_iw\in W_i^\perp$. On the other hand there are $V,\omega\in \R^3$ such that $P_iw(x)=V+\omega\wedge (x-X_i)$ for all $x\in\overline{B_i}$. 
 
Also, using the vector rule $A\wedge (B \wedge C)=(A\cdot C)B-(A\cdot B)C$ we have:
 \begin{align*}
  \fint_{\partial B_i} (x-X_i)\wedge \bra{\omega\wedge (x-X_i)}\dx{S}=\frac{2R^2}{3} \omega,
 \end{align*}
 and therefore
 \begin{align*}
  0&=\fint_{\partial B_i} (x-X_i)\wedge \bra{w-P_iw}\dx{S}=\fint_{\partial B_i} (x-X_i)\wedge w\dx{S}-\frac{2R^2}{3} \omega.
 \end{align*}
\end{proof}

As a consequence of this characterization we obtain a Poincar\'e inequality and a (first) Korn inequality and consequently a Korn-Poincar\'e inequality on $W_i^\perp$. 

\begin{lemma}\label{lemma:poincare}
 Let $r>0$ and $X\in \R^3$. Let $p\in [1,\infty]$ and let 
 \begin{align*}
  H^{1,p}_{X,r}\coloneqq \set{w\in W^{1,p}(B_r(X)): \int_{\partial B_r(X)} w\dx{S}=0}.
 \end{align*}
 For $p>1$ there is a constant $C>0$ that does not depend on $X$ or $r$ such that for all $w\in H^{1,p}_{X,r}$:
 \begin{align*}
  \norm{w}_{L^p(B_r(X))}\le Cr\norm{\nabla w}_{L^p(B_r(X))}
 \end{align*}
\end{lemma}

\begin{proof}[Proof of Lemma \ref{lemma:poincare}]
 By scaling and translation it is enough to prove the inequality for $X=0$ and $r=1$. It is well known, that for closed cones in which $\nabla w=0$ implies that $w=0$, such a Poincar\'e inequality holds. 
\end{proof}

\begin{lemma}\label{lemma:korn}
 Let $r>0$ and $X\in \R^3$. Let $p\in [1,\infty]$ and let 
 \begin{align*}
  H^{1,p}_{\sigma,X,r}\coloneqq \set{w\in W^{1,p}(B_r(X)): \int_{\partial B_r(X)} \!\!\!\!w\dx{S}=\int_{\partial B_r(X)}\!\!\!\! (x-X)\wedge w\dx{S}=0}.
 \end{align*}
 For $p\in (1,\infty)$ there is a constant $C>0$ that does not depend on $X$ or $r$, such that for all $w\in H^{1,p}_{\sigma,X,r}$:
 \begin{align}\label{eq:korn}
  \norm{\nabla w}_{L^p(B_r(X))}\le C\norm{ew}_{L^p(B_r(X))}
 \end{align}
\end{lemma}
\begin{proof}[Proof of Lemma \ref{lemma:korn}]
Again, we only need to prove the result for $X=0,r=1$, the case of general $X$ and $r$ follows by translation and rescaling. In the following we will omit the domain $B_1(0)$ in the spaces.
 
$H^{1,p}_{\sigma,0,1}$ is a closed cone. The proof of this Korn inequality is indirect and uses the same idea as the proof of general Poincar\'e inequalities on closed cones. For $w\in H^{1,p}_{\sigma,0,1}$ we have that $ew=0$ implies $w=0$. This follows from a combination of the well-known fact that $ew=0$ implies that $w$ is skew-symmetric affine and the boundary conditions in $H^{1,p}_{\sigma,0,1}$. 
 
 Now, for the sake of contradiction, assume that there is no $C>0$ such that \eqref{eq:korn} holds for all $w\in H^{1,p}_{\sigma,0,1}$. Then there is a sequence $w_k\in H^{1,p}_{\sigma,0,1}$ such that $\norm{\nabla w_k}_{L^p}\ge k \norm{ew_k}_{L^p}$. By rescaling we can arrange that $\norm{w_k}_{W^{1,p}}=1$ for all $k\in \N$. But then there is a subsequence (again denoted by $w_k$) and $w_\ast \in W^{1,p}$ such that $w_k\rightharpoonup w_\ast$ in $W^{1,p}$. Note that since $H^{1,p}_{\sigma,0,1}$ is convex we have that $w_\ast\in H^{1,p}_{\sigma,0,1}$. On the other hand we know that 
 \begin{equation*}
  \norm{ew_k}_{L^p}\le \frac{1}{k} \norm{\nabla w_k}_{L^p}\le \frac 1 k.
 \end{equation*} 
Hence $ew_k\to 0$ in $L^p$. Since at the same time $ew_k\rightharpoonup ew_\ast$ we get $ew_\ast=0$ . By our foregoing considerations this implies that $w_\ast=0$. \\
 
 By compact embedding we know that $w_k\to 0$ strongly in $L^p$. But to reach a contradiction we also need that the full gradient $\nabla w_k\to 0$ strongly, not only the symmetrized part. The key idea is, that by Korn's second inequality the gradients are already close to some constant skew-symmetric matrices for which we have strong compactness. \\
 
 Indeed, (\cite{KO88}, \S 2, Theorem 8) there exist matrices $A_k\in \R^{3\times 3}_{\skw}$ such that
 \begin{align*}
  \norm{\nabla w_k-A_k}_{L^p}\le C\norm{ew}_{L^p}\le C\frac 1 k.
 \end{align*}
 Since $\nabla w_k$ is bounded in $L^p$ this implies that the sequence $(A_k)_k$ must be bounded in $\R^{3\times 3}_{\skw}$. But then there is a subsequence (again denoted $A_k$) such that $A_k\to A_\ast$. Furthermore we can pick the subsequence in such a way that
 \begin{equation*}
  \abs{B_1(0)}^{1/p}\abs{A_k-A_\ast}\le \frac 1 k.
 \end{equation*} 
Then we have
 \begin{align*}
  \norm{\nabla w_k-A_\ast}_{L^p}\le \norm{\nabla w_k-A_k}_{L^p}+\norm{A_k-A_\ast}_{L^p}\le (C+1)\frac 1 k.
 \end{align*}
 
 Therefore, $\nabla w_k$ converges strongly to $A_\ast$ in $L^p$. But at the same time $\nabla w_k \rightharpoonup 0$ weakly in $L^p$. This yields $A_\ast=0$ and $\nabla w_k\to 0$ strongly in $L^p$ and $w_k\to 0$ strongly in $W^{1,p}$. This is a contradiction.
\end{proof}

Combining  Lemma \ref{lemma:poincare} and Lemma \ref{lemma:korn} as well as the Sobolev-embedding $W^{1,q}\hookrightarrow L^\infty$ for $q>3$ we obtain

\begin{cor}\label{cor:korn_poincare}
 Let $p\in (1,\infty]$. There is a constant $C>0$ that does not depend on $X$ or $r$ such that for all $w\in H^{1,p}_{\sigma,X,r}$:
 \begin{align*}
  \norm{w}_{L^p(B_r(X))}\le Cr\norm{ew}_{L^p(B_r(X))}.
 \end{align*}
\end{cor}

We know that elements of $w\in W_i^\perp$ solve the homogeneous Stokes equation outside $\overline{B_i}$. It is well-known that this coincides with minimizing the respective norm.

\begin{lemma}\label{lemma:minS}
 Let $s\in \dHs$ and $-\Delta s+\nabla p=0$ on $\R^3\setminus V$ for some closed $V$ with Lipschitz boundary. Then $s$ minimizes $\norm{ew}_{L^2(\R^3\setminus V)}$ among all $w\in \dHs$ with $w=s$ on $V$.
\end{lemma}

By choosing a suitable competitor we can therefore get estimates of the norm of $w\in W_i^\perp$ in terms of its values in $\overline{B_i}$. This competitor can be constructed using extension operators.

\begin{lemma}\label{lemma:extensionS}
There is a constant $C>0$ that does not depend on $X$ or $r$, and an extension operator $E_{X,r}:H^{1,2}_{\sigma,X,r}\to H_{\sigma,0}^1(B_{2r}(X))$ such that
 \begin{align*}
  \norm{\nabla E_{X,r}w}_{L^2(B_{2r}(X))}\le C \norm{ew}_{L^2(B_r(X))} \mbox{ for all }w\in H^{1,2}_{\sigma,X,r}.
 \end{align*}
\end{lemma}

\begin{proof}
 We can get $E_{X,r}$ from $E_{0,1}$ by translation and scaling without changing $C$. Let $E_{0,1}:H^1_{\sigma}(B_1(0))\to H_{\sigma,0}^1(B_{2}(0))$ be a continuous extension operator. This can be constructed from the usual extension operators by using the Bogowskii-operator (see \cite{Bog80}. Then we have
 \begin{align*}
  \norm{\nabla E_{0,1}w}_{L^2(B_2(0))}&\le C \norm{w}_{H^1(B_1(0))}\le C \norm{\nabla w}_{L^2(B_1(0))}\le C \norm{ew}_{L^2(B_1(0))},
 \end{align*}
 where the second to last inequality came from Lemma \ref{lemma:poincare} while the last is the Korn inequality from Lemma \ref{lemma:korn}. 
\end{proof}

Using $E_{X_i,R}w$ as a competitor in Lemma \ref{lemma:minS} and employing Lemma \ref{lemma:extensionS} we obtain

\begin{cor}\label{cor:norm_dHs}
 There is a constant $C>0$ such that for all $w\in W_i^\perp$
 \begin{align*}
  \norm{w}_{\dH}\le C \norm{ew}_{L^2(B_i)}.
 \end{align*}
\end{cor}

We can furthermore prove the following decay properties of the dipoles

\begin{lemma}\label{lemma:decayS}
 There is a constant $C>0$ such that for all $w\in W_i^\perp$ and for all $x\in \R^3\setminus B_{2R}(X_i)$ we have
 \begin{align}
  \abs{w(x)}&\le C \frac{R^{\frac 3 2}}{\abs{x-X_i}^2}\norm{w}_{\dH},\label{eq:decayS1}\\
  \abs{\nabla w(x)}&\le C \frac{R^{\frac 3 2}}{\abs{x-X_i}^3}\norm{w}_{\dH}.\label{eq:decayS2}
\end{align}
\end{lemma}
\begin{proof}
 The main idea of the proof is to use that the convolution of the fundamental solution is the inverse of the application of the Stokes operator. Therefore $w$ can be represented as the convolution of $Sw$ and the fundamental solution. For any function $w$ with compactly supported $Sw$ this gives a $\frac{1}{\abs{x}}$-like decay. The fact that $Sw$ integrates to zero due to the dipole property gives rise to the additional power in the decay. For details see Lemma 3.11 in \cite{Hof18}.
\end{proof}

\subsection{Convergence of the method of reflections}\label{subsec:convergence}

The ultimate goal of the method of reflections is, to prove that the symmetrized gradients of the approximations $v_k$ approach $0$ in $L^\infty$ inside the particles. If $w\in \dH$ has an uniformly bounded gradient in $B_i$ we can get the following estimates 

\begin{cor}\label{cor:decayS}
 There is a constant $C>0$ such that for all $w\in \dHs \cap W^{1,\infty}(B_i)$ and for all $x\in \R^3\setminus B_{2R}(X_i)$ we have 
 \begin{align}
  \abs{Q_i w(x)}\le C \frac{R^3}{\abs{x-X_i}^2}\norm{ew}_{L^\infty(B_i)}, \label{eq:decay_inftyS1}\\
  \abs{\nabla Q_i w(x)}\le C \frac{R^3}{\abs{x-X_i}^3}\norm{ew}_{L^\infty(B_i)}. \label{eq:decay_inftyS2}
 \end{align}
\end{cor}

\begin{proof}
 We know that $Q_iw=w-P_i w$ and since $P_i w$ is a rigid body motion on $B_i$ we have $eP_iw=0$ in $B_i$ and hence $eQ_iw=ew$ in $B_i$. Using that $Q_iw\in W_i^\perp$,  and combining Corollary \ref{cor:norm_dHs} with Lemma \ref{lemma:decayS} yields \eqref{eq:decay_inftyS1} and \eqref{eq:decay_inftyS2}.
\end{proof}

In order to prove the main approximation statement of this section we need some estimates for recurring sums (also see Lemma 2.1 of \cite{JO04} for the first two inequalities):

\begin{lemma}\label{lemma:sum}
There exists $C>0$ such that for any $x \in \R^3$ with $1\le i\le N$ such that $\abs{x-X_i}\le \abs{x-X_j}$ for all $1\le j\le N$ it holds that 
 \begin{align}
  \sum_{j\neq i}\frac{1}{\abs{x-X_j}}\le C \frac{N^{\frac 2 3}}{d}\le C N,\label{eq:sum4}\\
  \sum_{j\neq i}\frac{1}{\abs{x-X_j}^2}\le C \frac{N^{\frac 1 3}}{d^2} \le C N,\label{eq:sum5}\\
  \sum_{j\neq i}\frac{1}{\abs{x-X_j}^3}\le C \frac{\log N}{d^3}\le C N\log N,\label{eq:sum6}\\
  \sum_{j\neq i}\frac{1}{\abs{x-X_j}^4}\le C \frac{1}{d^4}\le C N^{\frac 4 3}.\label{eq:sum7}
 \end{align}
\end{lemma}

\begin{proof}
 Without loss of generality we assume $i=1$ and $X_1=0$. We order the balls in such a way that $\abs{X_1}\le \dots \le\abs{X_N}$. Since $d_{ij}\ge d$ the balls $B\bra{X_i,\frac{d}{2}}$ and $B\bra{X_j,\frac{d}{2}}$ do not intersect. Moreover for any $2\le i\le N$ we have
 \begin{align*}
  \bigcup_{j=1}^iB\bra{X_j,\frac{d}{2}}\subset B\bra{0,\frac{d}{2}+\abs{X_i}}\subset B(0,2\abs{X_i}).
 \end{align*}
We compare the left and the right volume to obtain
$  i\bra{\frac{d}{2}}^3\le (2\abs{X_i})^3$ and $\abs{X_i}\ge \frac{1}{4}d \;i^\frac{1}{3}$.
Now 
 \begin{align*}
  \sum_{i=2}^N\frac{1}{d_{1i}}&=\sum_{i=2}^N\frac{1}{\abs{X_i}}\le \frac{4}{d}\sum_{i=2}^N i^{-\frac{1}{3}}\le 
  \frac{4}{d}\int_0^N x^{-\frac{1}{3}}\dx{x}\le 12\frac{N^{1-\frac{1}{3}}}{d}\le C N.
 \end{align*}
 We deduce now \eqref{eq:sum4}  from this and the fact that $\abs{x-X_j}\ge \frac{1}{2}\abs{X_i-X_j}$. 
 The other estimates follow similarly. 
\end{proof}

We will also use a maximum modulus theorem for the Stokes equation:

\begin{lemma}[\cite{MRS99}]\label{lemma:max_mod}
 Let $\Omega\subset \R^3$ be an exterior domain (a domain with bounded complement) and assume that $g\in C^0(\Omega^C)$ satisfies
 $  \int_{\partial \Omega} g\cdot n\dx{S}=0.$

 Then, there is a constant $C>0$ that depends only on $\Omega$, such that the unique solution $u\in \dHs(\R^3)$ of the Dirichlet problem
 \begin{align*}
  -\Delta u+\nabla p&=0 \mbox{ in } \Omega,\\
  \Div u&=0 \mbox{ in } \Omega,\\
  u&=g \mbox{ on } \Omega^C,\\
 \end{align*}
 satisfies 
$  \norm{u}_{L^\infty}\le C\norm{g}_{L^\infty}$.
\end{lemma}

\begin{rem}
 We  only have to use this statement for $\Omega$ being the exterior of a ball. The constant is invariant under translation and scaling of the ball.
\end{rem}
 
We are now able to prove the convergence of the method of reflections.
%

\begin{thm}\label{thm:u_v1_diff}
 There is $\varepsilon>0$ such that if $\phi\log N<\varepsilon$, we have $v_k\to u$ in $\dH$ and $L^\infty(\R^3)$ for $k\to \infty$ and in particular
 
 \begin{align*}
 \norm{u-v_1}_{L^\infty}\le \phi o(1), \mbox{ as } N\to \infty.
 \end{align*}
\end{thm}

\begin{proof}
The proof is  in principle similar to the proof of Proposition 3.12 in \cite{Hof18} and we refer to that article for details. We explain briefly the strategy.

We first have to establish that for all $v_k$ we have $Pv_k=u$ where $P$ is the projection onto $W$. Using this we know that $\norm{u-v_k}_{\dH}\le \norm{w-v_k}_{\dH}$ for all $w\in W$. With the use of the extension operator $E$ we then construct suitable competitors in $W$ to obtain 
\begin{equation*}
\norm{u-v_k}^2_{\dH}\le C \sum_{i=1}^N R^3\norm{e v_k}^2_{L^\infty(\cup B_i)}\le C\phi \norm{e v_k}^2_{L^\infty(\cup B_i)}.
\end{equation*}
Using that $e v_k-e Q_iv_k=0$ on $B_i$ and the decay of $\nabla Q_jv_k$ for $j\neq i$ (Corollary \ref{cor:decayS}) together with Lemma \ref{lemma:sum} gives an estimate of the form $\norm{e v_{k+1}}_{L^\infty(\cup B_i)}\le C \phi \log N \norm{e v_k}_{L^\infty(\cup B_i)}$. By the smallness of $\phi\log N$ we obtain $v_k\to u$ in $\dH$. For the convergence in $L^\infty$ we can use the same decay techniques together with Lemma \ref{lemma:max_mod}. 

We just show that $e v \in L^\infty(\cup B_j)$ which is needed for the proof. For this we estimate
\begin{align*}
  \abs{\nabla v(x)}&=\abs{\int_{\R^3} (1-\sum_{k=1}^N \one_{B_k}(y))\nabla \Phi(x-y)f(y)\dx{y}}\\
  &\le C \int_{\R^3} \abs{f(y)}\frac{1}{\abs{x-y}^2}\dx{y}\\
  &\le C \norm{f}_{L^2}\bra{\int_{\R^3\setminus B_1(0)}\frac{1}{\abs{y}^{4}}\dx{y}}^{\frac 1 2}+C\norm{f}_{L^\infty}\int_{B_1(0)}\frac{1}{\abs{y}^2}\dx{y}\\
  &\le C \norm{f}_{L^2}+C\norm{f}_{L^\infty}
 \end{align*}
Therefore actually $\nabla v\in L^\infty(\R^3)$. 
\end{proof}

\subsection{The explicit dipole approximation}\label{subsec:explicit}

The dipole approximation $v_1$ is close enough to $u$. The next step is to show that $v_1$ is close to the explicit dipole approximation $\tilde{u}$ from \eqref{eq:tildeu}. For this we need a technical lemma that gives us additional decay for differences of terms from the fundamental solution.

\begin{lemma}\label{lemma:small_diff}
 There is a constant $C>0$ such that for $x,z\in \R^3$ with $\abs{x-X_i}\le \frac 1 2\abs{X_i-z}$ and all $n>l\ge 0$ and $k\neq i$ we have the following estimate:
 \begin{align*}
  \abs{\frac{(x-z)^{\otimes l}}{\abs{x-z}^n}-\frac{(X_i-z)^{\otimes l}}{\abs{X_i-z}^n}}\le C \frac{\abs{x-X_i}}{\abs{X_i-z}^{n-l+1}}.
 \end{align*}
 Here, for $y\in \R^3$, the tensor $y^{\otimes l}\in \R^{3^l}$ is given by $y^{\otimes l}_{i_1\dots i_l}=y_{i_1}\dots y_{i_l}$.
\end{lemma}

\begin{proof}
  We have $\abs{x-z}\ge \frac 1 2\abs{X_i-z}$. First we set $g(y)=\frac{(y-z)^{\otimes l}}{\abs{y-z}^n}$. Observe that for $y\in B_{\abs{x-X_i}}(X_i)$:
    \begin{align*}
   \abs{\nabla g(y)}\le C \frac{1}{\abs{y-z}^{n-l+1}}\le C \frac{1}{\abs{X_i-z}^{n-l+1}},
  \end{align*}
  since $\abs{y-z}\ge \frac 1 2 \abs{X_i-z}$. Then 
  \begin{align*}
   \abs{g(x)-g(X_i)}\le \abs{x-X_i}\int_0^1 \abs{\nabla g((1-t)X_i+tx)}\dx{t}\le C  \frac{\abs{x-X_i}}{\abs{X_i-z}^{n-l+1}}.
  \end{align*}
\end{proof}

The symmetrized gradient of the dipoles $d_i$ from \eqref{eq:di} and $Q^P_iv$ inside the balls is $ev(X_i)$ and $ev(x)$ respectively. This is why we need to control the oscillation of $\nabla v$ inside the balls. 

\begin{lemma}\label{lemma:nablav_diff}
 There is constant $C>0$ such that for all $x\in \R^3$ and $a=\abs{x-X_i}$
 \begin{equation}\label{eq:vosc}
  \abs{\nabla v(X_i)-\nabla v(x)}\le C\bra{ \phi +\phi^{\frac 1 4}+a+a^\alpha}.
 \end{equation}
  Recall that $\alpha$ is the H\"older exponent of $f$.
\end{lemma}

\begin{rem}\label{remark:R}
 In particular, taking $a<R$ we have
 \begin{align*}
  \norm{\nabla v(X_i)-\nabla v}_{L^\infty(B_i)}\le C\bra{ \phi+\phi^{\frac 1 4}+R+R^\alpha}\le o(1), \mbox{ as } N\to \infty.
 \end{align*}
 Recall that $\alpha$ is the H\"older exponent of $f$. The exponent $\frac 1 4$ is chosen for ease of notation. Any exponent smaller than $\frac 1 3$ works.
\end{rem}

\begin{proof}[Proof of Lemma \ref{lemma:nablav_diff}]
If $v$ was in $C^2$ the statement would be an easy consequence of the Taylor expansion. But we cannot expect $v$ to be two times differentiable since $f^N$ is zero on the particles and hence in general not continuous. Still, the function that is subtracted from $f$ to obtain $f^N$ is supported only on the particles and hence should have a contribution vanishing with $\phi$. This can be made clear by writing $\nabla v=\nabla \Phi\ast f^N$. For the part that remains we expect H\"older continuity, since $f$ is H\"older continuous. We know:
 \begin{equation*}
  \abs{\nabla v(X_i)-\nabla v(x)}=\abs{\int_{\R^3}(1-\sum_{k=1}^N \one_{B_k}(y))f(y)\bra{\nabla \Phi(X_i-y)-\nabla \Phi(x-y)}\dx{y}}.
 \end{equation*}
  We split the term into a part with the pure $f$ and a part that incorporates the characteristic functions of the particles. Let us first estimate the latter for $\abs{x}<2L$:
 \begin{align*}
  &\abs{\int_{\R^3}\sum_{k=1}^N \one_{B_k}(y) f(y)\bra{\nabla \Phi(X_i-y)-\nabla \Phi(x-y)}\dx{y}}\\
  &=\abs{\int_{B_L(0)}\sum_{k=1}^N \one_{B_k}(y) f(y)\bra{\nabla \Phi(X_i-y)-\nabla \Phi(x-y)}\dx{y}}\\
  &\le \norm{f}_{L^\infty} \norm{\sum_{k=1}^N \one_{B_k}}_{L^4}\norm{\nabla\Phi}_{L^{\frac 4 3}(B_{3L}(0))}\le C \phi^{\frac 1 4}.
 \end{align*}
 Here we used that $\abs{\nabla \Phi(x)}\le C \frac {1}{\abs{x}^2}\in L^{\frac 4 3}(B_{3L}(0))$. If $\abs{x}\ge 2L$ then $\abs{x-y}\ge L$ and we can estimate $\nabla \Phi$ in $L^\infty$ obtaining the $\phi$ term. For the other term and $\abs{x}<L$ we compute:
 \begin{align*}
  &\abs{\int_{\R^3}f(y)\bra{\nabla \Phi(X_i-y)-\nabla \Phi(x-y)}\dx{y}}\\
  &\le \norm{f}_{L^2}\norm{\frac{a}{\abs{X_i-\cdot}^3}}_{L^2(\R^3\setminus B_{3L}(0))}+\abs{\int_{B_{3L}(X_i)\cap B_{3L}(x)}(f(X_i-y)-f(x-y))\nabla \Phi(y)\dx{y}}\\
  &+\abs{\int_{B_{3L}(X_i)\setminus B_{3L}(x)}f(X_i-y)\nabla \Phi(y)\dx{y}}+\abs{\int_{B_{3L}(x)\setminus B_{3L}(X_i)}f(x-y)\nabla \Phi(y)\dx{y}}\\
  &\le C\bra{ a+a^\alpha[f]_\alpha \norm{\nabla \Phi}_{L^1(B_{3L}(X_i)\cap B_{3L}(x))}+a(3L)^2\norm{f}_{L^\infty}\frac{1}{L^2}}\le C\bra{ a+a^\alpha}.
 \end{align*}
Here we used Lemma \ref{lemma:small_diff} in the first step and H\"older continuity of $f$ as well as decay properties of $\nabla \Phi$. If $\abs{x}\ge L$ we can just use the $L^\infty$ bound on $\nabla v$ to estimate the difference by a constant. Combining all estimates we arrive at \eqref{eq:vosc}.
\end{proof}

\begin{cor}\label{cor:dipole_diff}
 For $x\in \R^3\setminus B_{2R}(X_i)$and $d_i$ from \eqref{eq:di} we have 
 \begin{align*}
  \abs{d_i(x)-Q_iv(x)}\le \frac{R^3}{\abs{x-X_i}^2}o(1), \mbox{ as } N\to \infty.
 \end{align*}
\end{cor}

\begin{proof}
 By computation we see that $d_i\in W_i^\perp$ and hence $d_i(x)-Q_iv\in W_i^\perp$. Applying Corollary \ref{cor:decayS} we obtain
 \begin{align*}
  \abs{d_i(x)-Q_iv(x)}\le C \frac{R^3}{\abs{x-X_i}^2} \norm{e d_i-e Q_iv}_{L^\infty(B_i)}.
 \end{align*}
 And for $y\in B_i$, with Lemma \ref{lemma:nablav_diff}, we have:
 \begin{align*}
  \abs{ed_i(y)-e Q_iv(y)}&= \abs{ev(X_i)-ev(y)}\le \abs{\nabla v(X_i)-\nabla v(y)}=o(1).
 \end{align*}
\end{proof}

It remains to use Corollary \ref{cor:dipole_diff} to prove that the difference between the whole dipole approximations $v_1$ and $\tilde{u}$ is small. Because of the singularities of the dipoles we are now forced to consider the domain $\Omega^N_{\delta}$ from \eqref{eq:omegadelta}.

\begin{lemma}\label{lemma:v1_u_diff}
 \begin{align*}
  \norm{v_1-\tilde{u}}_{L^\infty(\Omega^N_{\delta})}\le \phi o(1), \mbox{ as } N\to \infty.
 \end{align*}
\end{lemma}

\begin{proof}
 We have $ v_1-\tilde{u}=\sum_{i=1}^N\bra{ Q_iv-d_i}$. Take $x\in \Omega^N_{\delta}$ and let $X_i$ be the closest centre point. Then by Corollary \ref{cor:dipole_diff}
\begin{align*}
 \abs{Q_iv(x)-d_i(x)}\le \frac{R^3}{\abs{x-X_i}^2}o(1)\le R^3\delta^{-2}o(1)\le R^3No(1)\le \phi o(1).
\end{align*}
For the dipoles, that are further away, we use Corollary \ref{cor:dipole_diff} and Lemma \ref{lemma:sum} to get
\begin{align*}
 \abs{\sum_{j\neq i} \bra{Q_jv(x)-d_j(x)}}\le o(1)\sum_{j\neq i} \frac{R^3}{\abs{x-X_i}^2}\le R^3No(1)=\phi o(1).
\end{align*}
\end{proof}

We now want to get a similar estimate in $\R^3\setminus \Omega_{\delta}^N$, the region around the particles. Remark \ref{remark:R} gives us a hint what kind of closeness we can hope for. Let us imagine for a moment that both the first and the second derivative of $v$ are H\"older continuous (i.e. $v$ solves the problem with $f$ and not with $f^N$ as a right hand side) and hence bounded. This improves the (optimal) estimate from \ref{lemma:nablav_diff} only slightly to contain $R$ instead of $R^\alpha$ as one of the smallest terms. Even when multiplied by the $R$ from the Poincar\'e inequality this is, in general, not of the type $\phi o(1)$. So instead of an estimate in $L^\infty$ we aim for an estimate in $L^p$ for some $p\ge 1$. In order to estimate the difference $v_1-\tilde u$ in this space we only need 
to consider the difference of the dipoles originated at the closest particle because we have $L^\infty$ control over the remaining terms. So we find approximately:
\begin{align*}
 &\norm{v_1-\tilde u}_{L^p(\cup_{i=1}^N B_i)}\\
 &= \bra{\sum_{i=1}^N \int_{B_i} \abs{v_1(x)-\tilde u(x)}^p\dx{x}}^\frac 1 p\sim \bra{\sum_{i=1}^N \int_{B_i} \abs{Q_iv(x)-d_i(x)}^p\dx{x}}^\frac 1 p\\
 &\le C \bra{\sum_{i=1}^N \int_{B_i} \bra{R o(1)}^p\dx{x}}^\frac 1 p\le C\bra{\sum_{i=1}^N R^{3+p}}^\frac 1 p o(1)= \bra{N R^{3+p}}^\frac 1 p o(1)\\
 &= \phi^{\frac 1 3 + \frac 1 p}N^{-\frac 1 3}o(1), 
\end{align*}
where we used Lemma \ref{lemma:poincare} and Remark \ref{remark:R} in the third line. For this to be of type $\phi o(1)$ we need $p\le\frac 3 2$. Notice that on the other hand the dipoles decay like $\frac{1}{\abs{x}^2}$. This is only in $L^p(\R^3\setminus B_L(0))$ for $p>\frac 3 2$. Using the explicit form of the $o(1)$ term, it is possible to find a common $p=\frac 3 2+\varepsilon$ for which we have closeness on the whole space. 



\begin{lemma}\label{lemma:Lpdiff}
 Let $U\subset \R^3$ be of finite measure. For $p\in [1,\frac 3 2]$ it holds:
  \begin{align*}
  \norm{v_1-\tilde{u}}_{L^p(U)}\le \phi o(1).
 \end{align*}
\end{lemma}

\begin{proof}
 First of all note that by Lemma \ref{lemma:v1_u_diff} we only need to prove the statement for $U=\cup_{i=1}^N B_r(X_i)$. Let $x\in U$, then $x\in B_r(X_i)$ for one and only one $i$ because 
 of Assumption \eqref{assump2}. By the proof of Theorem \ref{lemma:v1_u_diff} we only need to consider the dipole $d_i-Q_iv$ because the other dipoles behave exactly the same as outside $B_r(X_i)$, giving $L^\infty$ estimates. If $x \in B_{2R}(X_i)$, by the maximum modulus theorem from \cite{MRS99} plus corollary \ref{cor:korn_poincare} as well as Lemma \ref{lemma:nablav_diff}:
  \begin{align*}
  \abs{Q^S_iv(x)-d^S_i(x)}&\le C R\abs{e Q^P_iv(x)-e d^P_i(x)}=R\abs{e v(x)-e v(X_i)}\\
  &\le R\abs{\nabla v(x)-\nabla v(X_i)}\le R o(1).
 \end{align*}
By the computation done before the lemma this bounds $\norm{v_1-\tilde u}_{L^p(\cup_{i=1}^N B_{2R}(X_i))}$ by $\phi o(1)$. If $r=2R$ we are done. Otherwise observe that for $\abs{x-X_i}\in (2R,\delta)$ we can use Corollary \ref{cor:dipole_diff} and get
\begin{align*}
 \norm{v_1-\tilde u}&_{L^p(\cup B_{\delta}(X_i)\setminus B_{2R}(X_i))}\\
 &\le o(1)\bra{N \int_{2R}^{\delta}R^{3p}\frac{1}{\abs{x}^{2p}}\dx{x}}^{\frac 1 p}+o(1)\phi\le o(1)\bra{N \delta^{3-2p} R^{3p}}^{\frac 1 p}+o(1)\phi\\
 &\le o(1) N^\frac 1 p R^3+o(1)\phi\le\phi o(1).
\end{align*}
This was the computation for $p\neq \frac 3 2$. If $p= \frac 3 2$ we get 
\begin{align*}
 \norm{v_1-\tilde u}_{L^p(\cup B_{\delta}(X_i)\setminus B_{2R}(X_i))}&\le C  \bra{\phi^\frac 1 4 +R^\alpha+R} N^\frac 1 p R^3 \bra{\log \delta-\log 2R} \\
 &=\phi N^{-\frac 1 3} \bra{-\log R}  \bra{\phi+\phi^\frac 1 4 +R^\alpha+R}=\phi o(1).
\end{align*}
Here we used $-N^{-\frac 1 3}\phi^{\frac 1 4}\log R\le C R^{-\frac 1 2} N^{-\frac 1 3}\phi^{\frac 1 4}\le C \phi^{-\frac 1 6}N^{-\frac 1 6}\phi^{\frac 1 4}=o(1)$.
\end{proof}

\section{Homogenization}\label{sec:homogenization}

\subsection{From the microscopic approximation to a homogenized equation}\label{subsec:hom}

The solution $v$ to the reference problem without particles involves the individual particles on the right hand side and is therefore not in a good form for the treatment of the limit problem. Therefore we first prove that $v$ is close to the solution $\hat{v}$ of the following problem:
\begin{align}
 -\Delta \hat{v}+\nabla p&=(1-\phi \rho)f \mbox{ in } \R^3,\label{eq:hvS1}\\
 \Div \hat{v}&=0 \mbox{ in } \R^3,\label{eq:hvS2}\\
 \hat{v}(x)&\to 0 \quad \text{as }\abs{x}\to \infty,\label{eq:hvS3}
\end{align}

\begin{lemma}\label{lemma:v_diff}
 For the solution $v$ of problem \eqref{eq:v1}-\eqref{eq:v3} and the solution $\hat{v}$ of problem \eqref{eq:hvS1}-\eqref{eq:hvS3} it holds that
 \begin{align*}
  \norm{v-\hat{v}}_{W^{1,\infty}(\Omega^N_\delta)}\le\phi o(1).
 \end{align*}
 Let $U\subset \R^3$ be of finite measure and $p\in [1,3]$. Then
  \begin{align*}
  \norm{v-\hat{v}}_{L^p(U)}\le\phi o(1).
 \end{align*}
\end{lemma}

\begin{proof}
 The idea of the proof is to represent $v,\hat{v}$ in terms of the fundamental solution. Since there is a $\phi$ in front of $\rho$, and $\rho^N$ (see \ref{eq:rhoN}) is close to $\rho$ 
 (see assumption \eqref{assump4}) in a weak sense (when convoluted with the fundamental solution), we can interchange $\rho$ by $\rho^N$. For given $x\in \R^3$ we can ignore terms at regions that are close to $x$ since they are small anyway. For regions further away from $x$, the number density $\rho^N$ looks approximately like the rescaled sum of the characteristic functions of the particles.
 
 We write $v,\hat{v}$ by means of the fundamental solution to see that
 \begin{align*}
 \abs{v(x)-\hat{v}(x)}=\abs{\int_{\R^3} \bra{\sum_{k=1}^N \one_{B_k}(y)-\phi\rho(y)}\Phi(x-y)f(y)\dx{y}}.
 \end{align*}
The proof comes in four parts. The first part shows that we can replace $\rho$ by $\rho^N$ and that particles close to $x$ can be ignored, the second part establishes the closeness of the functions in $L^\infty$ while the third part is concerned with the closeness of the gradients in $L^\infty$. In the last part the necessary $L^p$ estimates are shown.

\textbf{Part 1:} We can replace $\rho$ by $\rho^N$ since
\begin{align*}
  &\abs{\int_{\R^3} (\phi\rho^N(y)-\phi\rho(y))\Phi(x-y)f(y)\dx{y}}\\
  &= \phi \abs{\int_{B_{L+1}(0)} (\rho_N(y)-\rho(y))\Phi(x-y)f(y)\dx{y}}=\phi o(1),
\end{align*}
because $\rho^N\rightharpoonup \rho$ in $L^p(B_{L+1}(0))$ and $f\Phi(x-\cdot)\in L^q(B_{L+1}(0))$ where $q$ is the H\"older dual of $p$ and hence $q< \frac 3 2$.

Let $X_i$ be the closest centre point to $x$. Then we can ignore the $i$th term in the sum:
\begin{align*}
 \abs{\int_{\R^3} \one_{B_i}(y)\Phi(x-y)f(y)\dx{y}}\le C R^3 \frac{1}{\delta}\le C R^3 N^{\frac 1 2}=\phi N^{-\frac 1 2}=\phi o(1),
\end{align*}
and we can replace $\one_{B_k}$ by $\frac{4\pi}{3}R^3 \delta_{X_k}$ using Lemmas \ref{lemma:small_diff} and \ref{lemma:sum} as well as the H\"older continuity of $f$:
\begin{align*}
  &\abs{\int_{\R^3} \sum_{k\neq i} (\one_{B_k}(y)-\frac{4\pi}{3}R^3\delta_{X_k})\Phi(x-y)f(y)\dx{y}}\\
  &\le\sum_{k\neq i}\abs{\int_{B_k}\Phi(x-y)f(y)\dx{y}-\frac{4\pi}{3}R^3\Phi(x-X_k)f(X_k)}\\
  &=\sum_{k\neq i}\abs{\int_{B_k}\Phi(x-y)f(y)-\Phi(x-X_k)f(X_k)\dx{y}}\\
  &\le C\frac \phi N \sum_{k\neq i}\bra{\fint_{B_k}\abs{\Phi(x-y)f(y)-\Phi(x-X_k)f(y)}+\fint_{B_k}\abs{\Phi(x-X_k)f(y)-\Phi(x-X_k)f(X_k)}\dx{y}}\\
  &\le C \frac \phi N\sum_{k\neq i}\bra{\norm{f}_{L^\infty}\frac{R}{\abs{x-X_k}^2}+\frac{1}{\abs{x-X_k}}[f]_{C^{0,\alpha}}R^\alpha}\le C \phi (R+R^\alpha)=\phi o(1).
\end{align*}
Therefore, what is left to show is
\begin{align*}
  \abs{\frac 1 N\sum_{k\neq i} \frac{4\pi}{3}\Phi(x-X_k)f(X_k)-\int_{\R^3} \rho^N(y)\Phi(x-y)f(y)\dx{y}}= o(1).
\end{align*}
We can ignore the contributions by particles in the range $s$ :
\begin{align*}
 \abs{\int_{B_s(x)} \rho^N(y)\Phi(x-y)f(y)\dx{y}}&\le C \int_{B_s(x)} \frac{1}{\abs{x-y}}\dx{y}\le C s^2=o(1),\\
 \abs{\frac 1 N \sum_{k\neq i, \abs{x-X_k}\le s} \Phi(x-X_k)f(X_k)}&\le C \frac 1 N \sum_{k\neq i, \abs{x-X_k}\le s} \frac{1}{\abs{x-X_k}}\\
 &\le C \frac 1 N\frac 1 d \bra{\frac{s^3}{d^3}}^{\frac 2 3}\le C s^2=o(1).
\end{align*}
Here we used that $\rho^N$ is uniformly bounded and in the range $s$ there can only be a number of particles $\le C \frac{s^3}{d^3}$ and then applied Lemma \ref{lemma:sum} with $N=\frac{s^3}{d^3}$.

The above reasoning applies to particles in the range of $3s$ in the same way. This means we can ignore all cubes $A_j$ that intersect the boundary $\partial B_{s}(x)$ since they will be included in $B_{3s}(x)$ anyway. 

\textbf{Part 2:} Therefore estimating the difference above reduces to estimating the difference of appropriately grouped terms in the sum to its corresponding parts (the cube $A_j$) of the integral. This means we want to estimate
\begin{align*}
 \sum_{j:\dist(A_j,x)>s} \abs{\frac{4\pi}{3}  \frac{1}{Ns^3}\sum_{X_k\in A_j}\int_{A_j}f(X_k)\Phi(x-X_k)-f(y)\Phi(x-y)\dx{y}}=o(1). 
\end{align*}
Using H\"older-continuity of $f$ and Lemma \ref{lemma:small_diff} we have
\begin{align*}
 \abs{(f(X_k)-f(y))\Phi(x-X_k)}&\le C \frac{s^\alpha}{\abs{x-X_k}},\\
 \abs{f(y)\bra{\Phi(x-X_k)-\Phi(x-y)}}&\le C \frac{s}{\abs{x-X_k}^2}.
\end{align*}
Hence, using Lemma \ref{lemma:sum}:
\begin{align*}
 &\sum_{j:\dist(A_j,x)>s} \abs{\frac{4\pi}{3}  \frac{1}{Ns^3}\sum_{X_k\in A_j}\int_{A_j}f(X_k)\Phi(x-X_k)-f(y)\Phi(x-y)\dx{y}}\\
 &\le C\sum_{j:\dist(A_j,x)>s} \abs{\frac{4\pi}{3}  \frac{1}{N}\sum_{X_k\in A_j}\frac{s^\alpha}{\abs{x-X_k}}+\frac{4\pi}{3}  \frac{1}{N}\sum_{X_k\in A_j}\frac{s}{\abs{x-X_k}^2}}\\
 &\le C \frac{s^\alpha}{N}\sum_{k\neq i} \frac{1}{\abs{x-X_k}}+C\frac{s}{N}\sum_{k\neq i} \frac{1}{\abs{x-X_k}^2}\\
 &\le C \bra{s^\alpha+s}=o(1). 
\end{align*}
\textbf{Part 3:} In order to understand that the estimate holds for the gradient note that
 \begin{align*}
  \abs{\nabla v(x)-\nabla \hat{v}(x)}=\abs{\int_{\R^3} (\sum_{k=1}^N \one_{B_k}(y)-\phi\rho(y))\nabla \Phi(x-y)f(y)\dx{y}}.
 \end{align*}
We can now reproduce steps 1 and 2 from the proof above using the following facts i)-vi) in that order:
\\
\begin{enumerate}[leftmargin=*]
\item[i)] $\rho_N\rightharpoonup \rho$ in $L^p(B_{L+1}(0))$ and $f\nabla\Phi(x-\cdot)\in L^q(B_{L+1}(0))$, since $\frac 1 {\abs{x}^2}$ is $q$-integrable for $q<\frac 3 2$\\
\item[ii)] $\frac{1}{\delta^2N}\to 0$, and therefore
  \begin{align*}
   \abs{\int_{\R^3} \one_{B_i}(y)\nabla \Phi(x-y)f(y)\dx{y}}\le C R^3 \frac{1}{\delta^2}=R^3 \frac{1}{\delta^2N}N=\phi o(1).
  \end{align*}\\
\item[iii)] The appearing sums over third powers are well behaved:
  \begin{align*}
  \phi \frac 1 N\sum_{k\neq i} \frac{R}{\abs{x-X_k}^3}\dx{y}\le C \phi \log N R\le C \phi N^{\frac 1 3} R=\phi^{\frac 4 3} =\phi o(1);
  \end{align*}\\
\item[iv)] Terms is range $s$ can be disregarded:
 \begin{align*}
 \int_{B_{s}(x)} \frac{1}{\abs{x-y}^2}\dx{y}&\le C s=o(1),\\
 \frac 1 N \sum_{k\neq i, \abs{x-X_k}\le s} \frac{1}{\abs{x-X_k}^2}&\le C \frac 1 N\frac{1}{d^2} \bra{\frac{s_N^3}{d^3}}^{\frac 1 3}\le C s=o(1);
 \end{align*}\\
\item[v)] The sum over squares is a good as the sum over first powers:
 \begin{align*}
  \frac{s^\alpha}{N}\sum_{k\neq i} \frac{1}{\abs{x-X_k}^2}\le C s^\alpha=o(1); 
 \end{align*}
\item[vi)] The sum over the third powers can be controlled since $s$ approaches zero fast enough:
 \begin{align*}
  \frac{s}{N}\sum_{k\neq i} \frac{1}{\abs{x-X_k}^3}\le C s \log N=o(1). 
 \end{align*}
\end{enumerate}
This gives 
\begin{align*}
 \abs{\nabla v(x)-\nabla \hat{v}(x)}=\phi o(1).
\end{align*}

\textbf{Part 4:} In order to get the $L^p$ result simply notice that we can use the $L^\infty$ results everywhere even where $x\in B_r(X_i)$ as long as we did not use that $\abs{x-X_i}>r$. In fact this was used only once so that we have to look at the following term again when $x\in B_r(X_i)$:
\begin{align*}
 \abs{\int_{\R^3}\one_{B_i}(y)\Phi(x-y)f(y)\dx{y}}\le C \int_{B_i}\frac{1}{\abs{x-y}}\dx{y}.
\end{align*}
If $\abs{x-X_i}>2R$ then this is smaller than $CR^3\frac{1}{\abs{x-X_i}}$. If $\abs{x-X_i}\le 2R$ it scales like $R^2$. Integrating the $p$th power of the left hand side over the union of the $B_r(X_i)$ gives
\begin{align*}
 &\bra{\int_{\cup_{i=1}^N B_r(X_i)}\abs{\int_{\R^3}\one_{B_i}(y)\Phi(x-y)f(y)\dx{y}}^p\dx{x}}^\frac 1 p\\
 &\le C\bra{N \bra{R^{2p}R^3+R^{3p}\int_{2R}^{\delta}t^{-p+2}\dx{t}}}^\frac 1 p\le C \bra{N \bra{R^{2p+3}+R^{3p} \bra{\delta}^{-p+3}}}^\frac 1 p\\
 &\le C \bra{N^{\frac 1 p} R^{2+\frac 3 p}+R^3 N^{\frac 1 p}\delta^{-1+\frac 3 p}}\le\phi o(1).
\end{align*}
\end{proof}

Now we can establish the first closeness result for the solution of the homogenized equation.

\begin{lemma}\label{lemma:tu_hu_diffS}
 Let $\hat{u}$ be the solution to 

 \begin{align}
  -\Div\bra{\nabla \hat{u}+5 \phi \rho\, e\hat{v}}+\nabla p&=(1-\phi \rho)f,\label{eq:interm_eqS}\\
  \Div \hat{u}&=0.
 \end{align}
 and let $\tilde{u}$ be the explicit Stokes dipole approximation. Then we have
 \begin{align*}
  \norm{\tilde{u}-\hat{u}}_{L^\infty(\Omega^N_\delta)}\le\phi o(1). 
 \end{align*}
  Let $U\subset \R^3$ be of finite measure and $p\in [1,\frac 3 2]$. Then
  \begin{align*}
  \norm{\tilde{u}-\hat{u}}_{L^p(U)}\le\phi o(1).
 \end{align*}
\end{lemma}

\begin{proof}
 In principle we employ the same strategy as in the proof of Lemma \ref{lemma:v_diff}. We represent $\hat{u}$ in terms of the fundamental solution and then we use Lemma \ref{lemma:v_diff} as well as $\rho^N\rightharpoonup \rho$ to show, that the difference of the sum and the integral is small. 
 
We write $\hat{u}$ (componentwise) in terms of the fundamental solution:
 \begin{align}\label{eq:hatu_formulaS}
  \hat{u}_j(x)=\hat{v}_j(x)+\int_{\R^3} 5 \phi \rho(y) e\hat{v}(y)_{ki}\partial_k \Phi_{ij}(x-y)\dx{y}.
 \end{align}
 In order to show the closeness of $\tilde{u}$ to the representation from \eqref{eq:hatu_formulaS} let $x\in \Omega^N_\delta$ be given. We have
\begin{align*}
  \abs{\tilde{u}_j(x)-\hat{u}_j(x)}\le \abs{v_j(x)-\hat{v}_j(x)}+\abs{-\sum_{k=1}^N {d_k}_j(x)-\int_{\R^3} 5\phi\rho(y)e \hat{v}(y)_{ki}\partial_k \Phi_{ij}(x-y)\dx{y}}.
 \end{align*}
 Where ${d_k}_j$ is the $j$-th component of the explicit dipole at particle $k$ (see \eqref{eq:di}) and \emph{not} the distance. Taking into account Lemma \ref{lemma:v_diff} it remains to prove that
 \begin{align*}
  \abs{-\sum_{k=1}^N d_k(x)-\int_{\R^3} 5\phi\rho(y)e \hat{v}(y)\nabla\Phi(x-y)\dx{y}}= \phi o(1).
 \end{align*}
 The proof is divided in several parts. The first part shows that we can replace $\hat{v}$ by $v$. In the second part we show that the closest particle as well as the fast decaying parts of the dipoles can be ignored. The next part determines the explicit form of the gradient of the fundamental solution when applied to a symmetric, trace-free matrix. In the fourth part it is shown that we can replace $\rho$ by $\rho^N$ and ignore close particles. The fifth part establishes the closeness of the functions in $L^\infty$ while the last part is concerned with the $L^p$ result. 
 
\textbf{Part 1:} Using Lemma \ref{lemma:v_diff} we have 
 \begin{align*}
  &\abs{\int_{\R^3} 5\phi\rho(y)(ev(y)-e\hat{v}(y))\nabla\Phi(x-y)\dx{y}}\\
  &=\abs{\int_{B_L(0)\cap \Omega^N_\delta} 5\phi\rho(y)(ev(y)-e\hat{v}(y))\nabla\Phi(x-y)\dx{y}}\\
  &+\abs{\int_{B_L(0)}\sum_{k=1}^N \one_{B_{\delta}}(y) 5\phi\rho(y)(ev(y)-e\hat{v}(y))\nabla\Phi(x-y)\dx{y}}\\
  &\ls \phi^2 o(1)\int_{B_L(0)\cap \Omega^N_\delta} \frac{1}{\abs{x-y}^2}\dx{y}+C\phi \norm{\nabla \Phi}_{L^{\frac 4 3}(B_L(0))}\norm{\sum_{k=1}^N \one_{B_r}}_{L^4}\\
  &\le\phi^2 o(1)+C\phi \bra{N(\delta)^3}^{\frac 1 4}=\phi(\phi o(1)+o(1))=\phi o(1).
 \end{align*}
 Therefore it suffices to prove
  \begin{align*}
  \abs{-\sum_{k=1}^N d_k(x)-\int_{\R^3} 5\phi\rho(y)e v(y)\nabla\Phi(x-y)\dx{y}}= \phi o(1).
 \end{align*}
 
\textbf{Part 2:} Let $X_i$ be the closest centre point to $x$. Then we can ignore the $i$th term in the sum:
 \begin{align*}
  \abs{d_i(x)}\le C \frac{R^3}{\abs{x-X_i}^2}+C\frac{R^5}{\abs{x-X_i}^4}\le C R^3\delta^{-2}=CR^3 N \frac{1}{N\delta^2}=\phi o(1).
 \end{align*}
 Next we look at the fast decaying terms of $d_k$:
 \begin{align*}
  &\sum_{k\neq i} R^5\bra{\frac{ev(X_k)(x-X_k)}{\abs{x-X_k}^5}-\frac 5 2  \frac{(x-X_k)\bra{(x-X_k)\cdot ev(X_k)(x-X_k)}}{\abs{x-X_k}^7}}\\
  &\le C \sum_{k\neq i} \frac{R^5}{\abs{x-X_k}^4}\le C R^5 N^{\frac 4 3}=R\phi^{\frac 4 3}=\phi o(1).
 \end{align*}
 Thus we can ignore these terms and only consider the slowly decaying terms of the form 
 \begin{align*}
  \frac 5 2 R^3\bra{\frac{(x-X_k)\bra{(x-X_k)\cdot ev(X_k)(x-X_k)}}{\abs{x-X_k}^5}}.
 \end{align*}

\textbf{Part 3:} We now derive an expression for $ev(y)\nabla\Phi(x-y)$. We compute the derivative of the Oseen-Tensor to be 
\begin{align*}
 \partial_k \Phi_{ij}(x)&=\frac{1}{8\pi}\bra{-\frac{\delta_{ij}x_k}{\abs{x}^3}+\frac{\delta_{ik}x_j+\delta_{jk}x_i}{\abs{x^3}}-3\frac{x_ix_jx_k}{\abs{x}^5}}.
\end{align*}
Take any symmetric, trace free matrix $\epsilon$. Then
\begin{align*}
 (\epsilon \nabla\Phi)_j&\coloneqq\epsilon_{ki}\partial_k \Phi_{ij}(x)=\frac{1}{8\pi}\bra{-\frac{\epsilon_{ki}x_k}{\abs{x}^3}+\frac{\epsilon_{kk}x_j+\epsilon_{ij}x_i}{\abs{x^3}}-3\frac{\epsilon{ki}x_ix_jx_k}{\abs{x}^5}}\\
 &=-\frac{3}{8\pi}\frac{\epsilon{ki}x_ix_jx_k}{\abs{x}^5}=-\frac{3}{8\pi}\bra{\frac{x\bra{x\cdot \epsilon x}}{\abs{x}^5}}_j.
\end{align*}
Let us replace $R^3=\phi \frac{1}{N}$. Then we are left to show that
\begin{align*}
  &\left\lvert\phi \frac{1}{N}\sum_{k\neq i} \frac 5 2 \frac{(x-X_k)\bra{(x-X_k)\cdot ev(X_k)(x-X_k)}}{\abs{x-X_k}^5}\right.\\
  &\left.-\phi \int_{\R^3} \frac{15}{8\pi}\rho(y)\frac{(x-y)\bra{(x-y)\cdot ev(y) (x-y)}}{\abs{x-y}^5}\dx{y}\right\rvert= \phi \cdot  o(1).
\end{align*}

\textbf{Part 4:} With the same type of argument as in Part 1 of the proof of \ref{lemma:v_diff} we can replace $\rho$ by $\rho^N$, we can leave out a ball of size $s$ around $x$ in the integral and we can ignore the parts of the sum where $\abs{X_k-x}\le s$.
It remains to show:
\begin{align*}
  &\frac{1}{N}\sum_{k:\abs{x-X_k}> s} \frac{(x-X_k)\bra{(x-X_k)\cdot ev(X_k)(x-X_k)}}{\abs{x-X_k}^5}\\
  &-\int_{\R^3\setminus B_{s}(x)} \frac{3}{4\pi}\rho^N(y)\frac{(x-y)\bra{(x-y)\cdot ev(y) (x-y)}}{\abs{x-y}^5}\dx{y}= o(1).
  \end{align*}
We can employ the same reasoning as above to exclude all particles in the range of $3s$. This means we can ignore all cubes $A_j$ that intersect the boundary $\partial B_{s}(x)$ since they will eventually be included in $B_{3s}(x)$ anyway. 

\textbf{Part 5:} Therefore estimating the difference above reduces to estimating the difference of appropriately grouped terms in the sum to its corresponding parts (the cube $A_j$) of the integral. I.e.~we need to estimate 
\begin{align*}
  &\sum_{j:\dist(A_j,x)>s} |\frac{1}{N} \sum_{X_k\in A_j}\frac{(x-X_k)\bra{(x-X_k)\cdot ev(X_k)(x-X_k)}}{\abs{x-X_k}^5}\\
  &-\int_{A_j} \frac{3}{4\pi}\rho^N(y)\frac{(x-y)\bra{(x-y)\cdot ev(y) (x-y)}}{\abs{x-y}^5}\dx{y}|= o(1).
\end{align*}
Notice that the cubes with $A_j\cap B_{2L}(0)=\emptyset$ have no contribution since there $\rho^N=0$. 

Looking at one term of the sum we are left to estimate
\begin{align*}
 &\left\lvert\frac{1}{N} \sum_{X_k\in A_j}\frac{(x-X_k)\bra{(x-X_k)\cdot ev(X_k)(x-X_k)}}{\abs{x-X_k}^5}\right.\\
 &\left.-\int_{A_j} \frac{3}{4\pi}\rho^N(y)\frac{(x-y)\bra{(x-y)\cdot ev(y) (x-y)}}{\abs{x-y}^5}\dx{y}\right\rvert.
\end{align*}
We now use the definition of $\rho^N$ to write this as
\begin{align*}
&\left\lvert\frac{1}{N}\sum_{X_k\in A_j}\frac{(x-X_k)\bra{(x-X_k)\cdot ev(X_k)(x-X_k)}}{\abs{x-X_k}^5}\right.\\
&\left.-\int_{A_j}\frac{3}{4\pi}\frac{1}{N s^3}\frac{4\pi}{3}n(A_k) \frac{(x-y)\bra{(x-y)\cdot ev(y) (x-y)}}{\abs{x-y}^5}\dx{y}\right\rvert\\
=& \frac{1}{N}\left\lvert\sum_{X_k\in A_j}\left(\frac{(x-X_k)\bra{(x-X_k)\cdot ev(X_k)(x-X_k)}}{\abs{x-X_k}^5}\right.\right.\\
&\left.\left.-\frac{1}{s^3}\int_{A_j} \frac{(x-y)\bra{(x-y)\cdot ev(y) (x-y)}}{\abs{x-y}^5}\dx{y}\right)\right\rvert\\
\le & C \frac{1}{Ns^3}\sum_{X_k\in A_j}\left\lvert\int_{A_j} \frac{(x-X_k)\bra{(x-X_k)\cdot ev(X_k)(x-X_k)}}{\abs{x-X_k}^5}\right.\\
&\left.-\frac{(x-y)\bra{(x-y)\cdot ev(y) (x-y)}}{\abs{x-y}^5}\dx{y}\right\rvert. 
 \end{align*}
We can replace $ev(y)$ by $ev(X_k)$ in the integral since for the difference, by Lemma \ref{lemma:nablav_diff}, we have:
\begin{align*}\
 \abs{ev(y)-ev(X_k)}\le \abs{\nabla v(y)-\nabla v(X_k)}\le C\bra{s^\alpha+s+\phi^{\frac 1 4}}=o(1),
\end{align*}
and hence
\begin{align*}
&\left\lvert\int_{A_j}\frac{(x-y)\bra{(x-y)\cdot ev(X_k) (x-y)}}{\abs{x-y}^5}\dx{y}\right.\\
&\left.-\int_{A_j}\frac{(x-y)\bra{(x-y)\cdot ev(y) (x-y)}}{\abs{x-y}^5}\dx{y}\right\rvert\le o(1)\int_{A_j} \frac{1}{\abs{x-y}^2}\dx{y}.
\end{align*}
Since the number of particles in one $A_j$ is bounded by $Ns^3$, adding this up we obtain
\begin{align*}
 &\sum_{j:\dist(A_j,x)>s}\frac{1}{Ns^3}\sum_{X_k\in A_j}\left\lvert \int_{A_j}\frac{(x-y)\bra{(x-y)\cdot ev(X_k) (x-y)}}{\abs{x-y}^5}\dx{y}\right.\\
 &\left.-\int_{A_j}\frac{(x-y)\bra{(x-y)\cdot ev(y) (x-y)}}{\abs{x-y}^5}\dx{y}\right\rvert\\
 &\le \sum_{j:\dist(A_j,x)>s, n(A_j)\neq 0} o(1) \int_{A_j}\frac{1}{\abs{x-y}^2}\dx{y}\le o(1) \int_{B_{L+s}(0)}\frac{1}{\abs{x-y}^2}\dx{y}\le o(1).
\end{align*}
 Using Lemma \ref{lemma:small_diff} for $y,X_k\in A_j$, we obtain
\begin{align*}
 &\frac{1}{Ns^3}\sum_{X_k\in A_j}\left\lvert\int_{A_j} \frac{(x-X_k)\bra{(x-X_k)\cdot ev(X_k)(x-X_k)}}{\abs{x-X_k}^5}\right.\\
 &\left.-\frac{(x-y)\bra{(x-y)\cdot ev(X_k)(x-y)}}{\abs{x-y}^5}\right\rvert\\
 &\le C \frac{1}{Ns^3}\sum_{X_k\in A_j}\abs{\int_{A_j} \norm{\nabla v}_{L^\infty}\frac{s}{\abs{x-X_k}^3}}\le C \frac{s}{N}\sum_{X_k\in A_j}\frac{1}{\abs{x-X_k}^3}.
\end{align*}
Summing up over $j$ gives
\begin{align*}
 &\sum_{j:\dist(A_j,x)>s} \frac{1}{Ns^3}\sum_{X_k\in A_j}\abs{\int_{A_j} \nabla v(X_k)\bra{\frac{x-X_k}{\abs{x-X_k}^3} -\frac{x-y}{\abs{x-y}^3}}}\\
 &\le C \frac{s}{N}\sum_{k\neq i}\frac{1}{\abs{x-X_k}^3}\le C \frac{s}{N}\frac{\log N}{d^3}\le C s\log N=o(1).
\end{align*}

\textbf{Part 6:} As in Part 4 of the proof of Lemma \ref{lemma:v_diff} we can use the $L^\infty$ results everywhere even where $x\in B_r(X_i)$ as long as we do not use that $\abs{x-X_i}>r$. This was only used once so that we have to look at $d_i(x)$ again when $x\in B_r(X_i)$: If $\abs{x-X_i}>R$, this is smaller than $CR^3\frac{1}{\abs{x-X_i}^2}$. If $\abs{x-X_i}\le R$ it scales like $R$. Integrating the $p$th power of this over the union of the $B_r(X_i)$ gives
\begin{align*}
 \bra{\sum_{i=1}^N \int_{B_r(X_i)}\abs{d_i(x)}^p\dx{x}}^\frac 1 p&\le C\bra{N \bra{R^pR^3+R^{3p}\int_{R}^{dN^{-\beta}}t^{-2p+2}\dx{t}}}^\frac 1 p\\
 &\le C \bra{N \bra{R^{p+3}+R^{3p} \bra{dN^{-\beta}}^{-2p+3}}}^\frac 1 p\\
 &\le  CN^{\frac 1 p} R^{1+\frac 3 p}+CR^3 N^{\frac 2 3 +2\beta-\frac 3 p \beta}\le\phi o(1).
\end{align*}
\end{proof}

\subsection{Passage to the Stokes equation with variable viscosity}\label{subsec:passage}

In order to obtain the final result we want to replace the $\hat{v}$ in equation \eqref{eq:interm_eqS} by $\hat{u}$. First we establish a regularity lemma: 

\begin{lemma}\label{lemma:reg_lemma}
 There is a constant $C>0$ such that the following holds. Let $g\in L^2(\R^3)$ be compactly supported in $B_{2L}(0)$. Let $w\in \dHs$ solve
 \begin{align*}
  -\Delta w+\nabla p&=g \mbox{ in } \R^3,\\
  \Div w&=0 \mbox{ in } \R^3.
 \end{align*}
  Then, $w\in L^\infty(\R^3)$ and $\norm{w}_{L^\infty(\R^3)}\le C\norm{g}_{L^2(\R^3)}$.
\end{lemma}

\begin{proof}
  We apply the fundamental solution to write
 \begin{align*}
  \abs{w(x)}&=\abs{\int_{R^3} \Phi(x-y)g(y)\dx{y}}=\abs{\int_{B_L(0)} \Phi(x-y)g(y)\dx{y}}\\
  &\le C \int_{B_L(0)} \frac{1}{\abs{x-y}}\abs{g(y)}\dx{y}\le C \norm{g}_{L^2}\norm{\frac{1}{\abs{y}}}_{L^2(B_{2L}(x))}\le C \norm{g}_{L^2}.
 \end{align*}
\end{proof}

Now we establish existence and estimates for the final equation:

\begin{lemma}
 There is a constant $C>0$ such that the following holds. The equation
 \begin{align*}
  -\Div\bra{\bra{2+5\phi \rho}e\bar{u}}+\nabla p&=(1-\phi \rho)f \mbox{ in } \R^3\\
  \Div \bar{u}&=0 \mbox{ in } \R^3,
 \end{align*}
 has a solution in $\dHs$ and for small $\phi$ we have $\norm{\nabla \bar{u}}_{L^2}\le \norm{f}_{L^\frac{6}{5}}$. Moreover, the gradient of the solution satisfies $\nabla \bar{u}\in H^1(\R^3)$. The estimate for $\nabla^2 \bar u$ is given by
 \begin{align*}
  \norm{\nabla^2 \bar u}_{L^2}\le \bra{ \phi \norm{\nabla \rho}_{L^\infty}\norm{f}_{L^{\frac 6 5}}+\bra{1+\phi\norm{\rho}_{L^\infty}}\norm{f}_{L^2}}.
 \end{align*}
\end{lemma}

\begin{proof}
 Consider the weak formulation 
 \begin{align}
 \int_{\R^3}\bra{2+5\phi \rho}e \bar{u}e\varphi\dx{x}=\int_{\R^3}(1-\phi \rho)f\varphi\dx{x},\label{eq:weak_baru}
 \end{align}
 where $\varphi\in \dHs$. Existence follows from Lax-Milgram theorem. For $\phi\le \norm{\rho}_{L^\infty}^{-1}$ we get the estimate for the gradient by setting $\varphi=\bar{u}$ and estimating the right hand side like
  \begin{align*}
 \int_{\R^3}(1-\phi \rho)f\varphi\le \norm{f}_{L^\frac 6 5}\norm{\bar u}_{L^6}\le \norm{f}_{L^\frac 6 5}\norm{\nabla \bar u}_{L^2}.
 \end{align*}
 The estimate on the second derivative is obtained by rewriting the weak formulation as
 \begin{align*}
  \int_{\R^3}\nabla \bar{u} \nabla \varphi\dx{x}+\int_{\R^3}5\phi \rho e \bar{u}e\varphi\dx{x}=\int_{\R^3}(1-\phi \rho)f\varphi\dx{x},
 \end{align*}
 and then testing with difference quotients $\varphi=-D^{-h}_k D^h_k \bar{u}$, where for any function $g$ $D^h_kg(x)=\frac 1 h \bra{g(x+he_k)-g(x)}$.
\end{proof}

\begin{lemma}\label{lemma:hu_bu_diff}
 There is a constant $C>0$ such that the following holds. The weak solutions in $\dHs$ to the equations 
 \begin{align}
  -\Div\bra{2e\hat{u}+5 \phi \rho e\hat{v}}+\nabla p=(1-\phi \rho)f,\\
  -\Div\bra{\bra{2+5 \phi \rho}e\bar{u}}+\nabla p=(1-\phi \rho)f,\label{eq:ubar}
 \end{align}
 differ on scale $\phi^2$, i.e. $\norm{\hat{u}-\bar{u}}_{L^\infty(\R^3)}\le C \phi^2$.
\end{lemma}

\begin{proof}
 By subtracting equation \eqref{eq:hvS1} from equation \eqref{eq:ubar} we obtain:
 \begin{align*}
  -\Div\bra{2e\bar u-2e \hat{v}+5\phi \rho e\bar u}+\nabla p=0.
 \end{align*}
 Hence, for the difference $\bar u -\hat v$, we get: 
 \begin{align}\label{laplace_u_v}
  -\Delta \bra{\bar u -\hat v}+\nabla p=\phi\Div\bra{5\rho e\bar u}.
 \end{align}
 Testing with $\bar u -\hat v$ gives
 \begin{align} \label{estimate_u_vS}
  \norm{\nabla \bar u -\nabla \hat v}_{L^2}\le 5\phi\norm{\rho}_{L^\infty}\norm{\nabla \hat u}_{L^2}\le C \phi \norm{\rho}_{L^\infty}\norm{f}_{L^\frac{6}{5}}.
 \end{align}
 On the other hand we know that $\nabla \bar u\in H^1$ and by the same argument $\nabla \hat v\in H^1$ so that we can test equation \eqref{laplace_u_v} by $-\Delta \bra{\bar u -\hat v}$ in order to obtain
 \begin{align*}
  \norm{\nabla^2 \bra{\bar u-\hat v}}_{L^2}^2&\le C \phi \norm{\nabla^2\bra{\bar u-\hat v}}_{L^2}\norm{\rho}_{W^{1,\infty}}\norm{\nabla \bar u}_{H^1},\\
  \norm{\nabla^2 \bra{\bar u-\hat v}}_{L^2}&\le C \phi \norm{\rho}_{W^{1,\infty}}\bra{1+\phi \norm{\rho}_{W^{1,\infty}}}\bra{\norm{f}_{L^2}+\norm{f}_{L^{\frac 6 5}}}.
 \end{align*}
 This proves that $\norm{\nabla\bra{\bar u-\hat v}}_{H^1}\le C \phi$. 
  Now we subtract the equations for $\bar u$ and $\hat u$ to obtain for the difference $w=\bar u -\hat u$:
  \begin{align*}
  -\Div\bra{\nabla w+5\phi \rho \bra{\nabla \bar u-\nabla \hat v}}+\nabla p=0.
 \end{align*}
 This means that
 \begin{align*}
  -\Delta w+\nabla p=\Div\bra{5\phi \rho \bra{\nabla \bar u-\nabla v}}.
 \end{align*}
 The right hand side is compactly supported in $B_{2L}(0)$ and in $L^2$. By Lemma \ref{lemma:reg_lemma} this means that
 \begin{align*}
  \norm{w}_{L^\infty}&\le C \phi \norm{\rho}_{W^{1,\infty}} \norm{\nabla \bar u-\nabla v}_{H^1}.\\
  &\le C \phi^2.
 \end{align*}
\end{proof}

\begin{proof}[Proof of Theorem \ref{thm:mainS}]
 The statement follows by combining Theorem \ref{thm:u_v1_diff}, Lemma \ref{lemma:v1_u_diff}, Lemma \ref{lemma:Lpdiff}, Lemma \ref{lemma:tu_hu_diffS} and Lemma \ref{lemma:hu_bu_diff}. Note that we do not need a separate $L^p$ statement in Theorem \ref{thm:u_v1_diff} and in \ref{lemma:hu_bu_diff} since we have control over the $L^\infty$ norm of the difference on the whole space.
\end{proof}

\section*{Aknowledgements}

We thank Juan Vel\'azquez for suggesting to look at the dipole approximations as well as for many fruitful discussions and Richard H\"ofer for his hint on how to use the method of reflections to its full potential.

The authors acknowledge support through the CRC 1060 The mathematics of emergent effects at the University of Bonn that is funded through the German Science Foundation (DFG).

\bibliographystyle{alpha}
\bibliography{Einstein}
\end{document}